\documentclass[12pt]{amsart}
\usepackage{amsmath, amssymb, amsfonts}
\usepackage[all]{xypic}
\usepackage[pdftex]{graphicx}
\usepackage{stmaryrd}

\theoremstyle{plain}
\newtheorem{theorem}{Theorem}[section]
\newtheorem{lemma}[theorem]{Lemma}
\newtheorem{corollary}[theorem]{Corollary}
\newtheorem{prop}[theorem]{Proposition}

\newtheorem{conj}[theorem]{Conjecture}

\theoremstyle{remark}

\newtheorem{remark}[theorem]{Remark}

\newtheorem{example}[theorem]{Example}

\newtheorem*{note*}{Note}
\newtheorem*{remark*}{Remark}
\newtheorem*{example*}{Example}

\theoremstyle{definition}
\newtheorem*{definition*}{Definition}
\newtheorem*{hypothesis*}{Hypothesis}
\newtheorem*{assumptions*}{Assumptions}
\newtheorem{definition}[theorem]{Definition}

\newcommand{\Z}{\mathbb{Z}}

\newcommand{\Q}{\mathbb{Q}}
\newcommand{\C}{\mathbb{C}}
\newcommand{\N}{\mathbb{N}}
\newcommand{\F}{\mathbb{F}}
\newcommand{\Aff}{\mathrm{Aff}}

\newcommand{\Ann}{\mathrm{Ann}}
\newcommand{\Aut}{\mathrm{Aut}}
\newcommand{\Gal}{\mathrm{Gal}}

\newcommand{\GL}{\mathrm{GL}}

\newcommand{\cl}{\mathrm{cl}}

\newcommand{\nr}{\mathrm{nr}}

\newcommand{\Hom}{\mathrm{Hom}}

\newcommand{\res}{\mathrm{res}}

\newcommand{\Irr}{\mathrm{Irr}}
\newcommand{\ind}{\mathrm{ind}}
\newcommand{\PMod}{\mathrm{PMod}}

\newcommand{\Spec}{\mathrm{Spec}}
\newcommand{\ram}{\mathrm{ram}}
\newcommand{\ab}{\mathrm{ab}}
\newcommand{\Br}{\mathrm{Br}}

\numberwithin{equation}{section}

\newcommand{\Fitt}{\mathrm{Fitt}}
\newcommand{\mal}{^{\times}}
\newcommand{\et}{\mathrm{\acute{e}t}}
\newcommand{\fl}{\mathrm{fl}}

\newcommand{\perf}{\mathrm{perf}}
\newcommand{\tor}{_{\mathrm{tor}}}
\newcommand{\Det}{\mathrm{Det}}
\newcommand{\aug}{\mathrm{aug}}

\newcommand{\ord}{\mathrm{ord}}

\newcommand{\Ext}{\mathrm{Ext}}

\newcommand{\Hyp}{\mathrm{Hyp}}

\title[On the non-abelian Brumer--Stark conjecture and the EIMC]{On the non-abelian Brumer--Stark conjecture\\ and the equivariant Iwasawa main conjecture}

\author{Henri Johnston}
\address{
Department of Mathematics\\
University of Exeter\\
Exeter\\
EX4 4QF\\
U.K.
}
\email{H.Johnston@exeter.ac.uk}
\urladdr{http://emps.exeter.ac.uk/mathematics/staff/hj241}

\author{Andreas Nickel}
\address{Universit\"{a}t Duisburg--Essen\\
    Fakult\"{a}t f\"{u}r Mathematik\\
    Thea-Leymann-Str. 9\\
    45127 Essen\\
    Germany}
\email{andreas.nickel@uni-due.de}
\urladdr{https://www.uni-due.de/$\sim$hm0251/english.html}

\date{Version of 1st August 2018}
\begin{document}

\maketitle

\begin{abstract}
We show that for an odd prime $p$, the $p$-primary parts of refinements of
the (imprimitive) non-abelian Brumer and Brumer--Stark conjectures are
implied by the equivariant Iwasawa main conjecture (EIMC) for totally real fields.
Crucially, this result does not depend on the vanishing of the relevant Iwasawa $\mu$-invariant.
In combination with the authors' previous work on the EIMC, 
this leads to unconditional proofs of the non-abelian Brumer and Brumer--Stark conjectures in many new cases.
\end{abstract}


\section{Introduction}

Let $K$ be a totally real number field and let $L$ be a CM field such that $L/K$ is a finite Galois extension; in this article,
such an extension will be called a (finite Galois) CM-extension. Let $G=\Gal(L/K)$.
To each finite set $S$ of places of $K$ containing all the archimedean places, one can associate a so-called `Stickelberger element'
$\theta_{S}(L/K)$ in the centre of the complex group algebra $\C[G]$. 
This element is constructed from values at $s=0$ of $S$-truncated Artin $L$-functions attached to the complex characters of $G$.
Let $\mu_{L}$ and $\cl_{L}$ denote the roots of unity and the class group of $L$, respectively.
Assume further that $S$ contains the set of all finite primes of $K$ that ramify in $L/K$.

Now suppose that $G$ is abelian.
It was independently shown in \cite{MR524276, MR579702, MR525346} that we have the containment
\[
\Ann_{\Z[G]} (\mu_{L}) \theta_{S}(L/K) \subseteq \Z[G].
\]
Moreover, Brumer's conjecture asserts that in fact 
\[
\Ann_{\Z[G]} (\mu_{L}) \theta_{S}(L/K) \subseteq \Ann_{\Z[G]}(\cl_{L}).
\]
The Brumer--Stark conjecture is a refinement of Brumer's conjecture that not only asserts that the class of a given ideal 
is annihilated in $\cl_{L}$, so it becomes a principal ideal, but also gives information about a generator of that ideal. 

In the case that $G$ is non-abelian, the second named author \cite{MR2976321} formulated generalisations of 
the Brumer and Brumer--Stark conjectures and of the so-called strong Brumer--Stark property. 
(Independently, Burns \cite{MR2845620} formulated non-abelian versions of the Brumer and Brumer--Stark conjectures in even greater generality.)
The extension $L/K$ satisfies the strong Brumer--Stark property if certain Stickelberger elements are contained in
the (non-commutative) Fitting invariants of corresponding ray class groups. 
It is important to note that this property does not hold in general, even in the case that $G$ is abelian,
as follows from results of Greither and Kurihara \cite{MR2443336}. 
If this property does hold, however, it also implies the validity of the (non-abelian) Brumer and 
Brumer--Stark conjectures.
The main result of this article is that for an odd prime $p$,
the relevant case of the equivariant Iwasawa main conjecture (EIMC) for totally real fields implies 
the $p$-primary part of a dual version of the strong Brumer--Stark property under the hypotheses
that $S$ contains all the $p$-adic places of $K$ and that a certain identity between complex and $p$-adic Artin $L$-functions at $s=0$ holds. 
This identity is conjecturally always true and, in particular, is satisfied when $G$ is monomial 
(i.e.\ every complex irreducible character of $G$ is induced from a one-dimensional character of a subgroup);
moreover, every metabelian or supersoluble finite group is monomial.

When the relevant classical Iwasawa $\mu$-invariant vanishes, the above result on the dual version 
of the strong Brumer--Stark property has been already established by the second named author \cite{MR3072281};
this result was in turn a non-abelian generalisation of work of Greither and Popescu \cite{MR3383600}.
(A weaker version of the result of  \cite{MR3383600} specialised to the setting of Brumer's conjecture was previously shown by Nguyen Quang Do \cite{MR2211312}.)
The main reason for the assumption of the vanishing of the $\mu$-invariant in \cite{MR3383600, MR3072281} is, of course, 
to ensure the validity of the EIMC. 
However, both articles use a version of the EIMC involving the Tate module of a certain Iwasawa-theoretic abstract
$1$-motive, which requires the vanishing of $\mu$ even for its formulation. This formulation is inspired by Deligne's theory
of $1$-motives \cite{MR0498552} and previous work of Greither and Popescu \cite{MR2899958}
on the Galois module structure of $p$-adic realisations of Picard $1$-motives.
Our new approach is different to, though partly inspired by, the approaches in \cite{MR3383600, MR3072281}.
More precisely, we reinterpret certain well-known exact sequences involving ray class groups in terms of
\'{e}tale and flat cohomology. Taking direct limits along the cyclotomic $\Z_{p}$-extension of $L$,
this allows us to establish a concrete link between the canonical complex
occurring in the EIMC and certain ray class groups. The theory of non-commutative Fitting invariants
then plays a crucial role in the Iwasawa co-descent.\\

This article is organised as follows. 
In \S \ref{sec:algebraic-prelims} we review some algebraic background material.
In particular, we discuss non-commutative Fitting invariants, which were introduced by the second named author in \cite{MR2609173} 
and were further developed by both the present authors in \cite{MR3092262}. 
In \S \ref{sec:Brumer--Stark} we recall the statements of the non-abelian Brumer and Brumer--Stark conjectures and show that a dual
version of the strong Brumer--Stark property implies both of these conjectures.
Then in \S \ref{sec:EIMC} we recall a reformulation of the EIMC introduced in \cite{MR3749195}.
The material presented up until this point then allows us to state the main theorem of this article (discussed above) in 
\S \ref{sec:statement-of-main-theorem-and-corollary}.
The next three sections are then devoted to the proof of this result. 
First we present further auxiliary results on Iwasawa algebras and the EIMC in 
\S \ref{subsec:iwasawa-algs-and-mcs-revisited}.
We then complete the proof in \S \ref{sec:proof-finite-level} and \S \ref{sec:proof-infinite-level} 
by working with complexes at the finite and infinite levels, respectively.
In \S \ref{sec:hybrid-and-frobenius} we recall the notion of hybrid $p$-adic group rings introduced in \cite{MR3461042}.
This notion was further developed in \cite{MR3749195} where it played a key role in obtaining the first unconditional proofs of the EIMC in cases where the vanishing of the relevant $\mu$-invariant is not known.
Finally, in \S \ref{sec:unconditional} we combine these results on the EIMC with the 
main result of this article to give unconditional proofs of the non-abelian Brumer and Brumer--Stark 
conjectures in many new cases. 

\subsection*{Acknowledgements}
It is a pleasure to thank Lennart Gehrmann, Cornelius Greither, Andreas Langer, Cristian Popescu, 
J\"urgen Ritter, Al Weiss and Malte Witte for helpful discussions and correspondence.
The first named author acknowledges financial support provided by EPSRC First Grant EP/N005716/1 `Equivariant Conjectures in Arithmetic'.
The second named author acknowledges financial support provided by the DFG within the Collaborative Research Center 701
`Spectral Structures and Topological Methods in Mathematics'.

\subsection*{Notation and conventions}
All rings are assumed to have an identity element and all modules are assumed
to be left modules unless otherwise  stated. We fix the following notation:

\medskip

\begin{tabular}{ll}
$R^{\times}$ & the group of units of a ring $R$\\
$\zeta(R)$ & the centre of a ring $R$\\
$\Ann_{R}(M)$ & the annihilator of the $R$-module $M$\\
$M_{m \times n} (R)$ & the set of all $m \times n$ matrices with entries in a ring $R$\\
$\zeta_{n}$ & a primitive $n$th root of unity\\
$K_{\infty}$ & the cyclotomic $\Z_{p}$-extension of the number field $K$\\
$\mu_{K}$ & the roots of unity of a field $K$\\
$\cl_{K}$ & the class group of a number field $K$ \\
$K^{\mathrm{c}}$ & an algebraic closure of a field $K$ \\
$K^{+}$ & the maximal totally real subfield of a field $K$ embeddable into $\C$\\
$\Irr_{F}(G)$ & the set of $F$-irreducible characters of the (pro)-finite group $G$\\
& (with open kernel) where $F$ is a field of characteristic $0$\\
$\check{\chi}$ & the character contragredient to $\chi$
\end{tabular}

\section{Algebraic Preliminaries}\label{sec:algebraic-prelims}

\subsection{Algebraic $K$-theory}\label{subsec:K-theory}
Let $R$ be a noetherian integral domain with field of fractions $E$.
Let $A$ be a finite-dimensional semisimple $E$-algebra and let $\mathfrak{A}$ be an $R$-order in $A$.
Let $\PMod(\mathfrak{A})$ denote the category of finitely generated projective (left) $\mathfrak{A}$-modules.
We write $K_{0}(\mathfrak{A})$ for the Grothendieck group of $\PMod(\mathfrak{A})$ (see \cite[\S 38]{MR892316})
and $K_{1}(\mathfrak{A})$ for the Whitehead group (see \cite[\S 40]{MR892316}).
Let $K_{0}(\mathfrak{A}, A)$ denote the relative algebraic $K$-group associated to the ring homomorphism
$\mathfrak{A} \hookrightarrow A$.
We recall that $K_{0}(\mathfrak{A}, A)$ is an abelian group with generators $[X,g,Y]$ where
$X$ and $Y$ are finitely generated projective $\mathfrak{A}$-modules
and $g:E \otimes_{R} X \rightarrow E \otimes_{R} Y$ is an isomorphism of $A$-modules;
for a full description in terms of generators and relations, we refer the reader to \cite[p.\ 215]{MR0245634}.
Moreover, there is a long exact sequence of relative $K$-theory (see \cite[Chapter 15]{MR0245634})
\begin{equation}\label{eqn:long-exact-seq}
K_{1}(\mathfrak{A}) \longrightarrow K_{1}(A) \stackrel{\partial}{\longrightarrow} K_{0}(\mathfrak{A}, A)
\stackrel{\rho}{\longrightarrow} K_{0}(\mathfrak{A}) \longrightarrow K_{0}(A).
\end{equation}
The reduced norm map $\nr = \nr_{A}: A \rightarrow \zeta(A)$ is defined componentwise on the Wedderburn decomposition of $A$
and extends to matrix rings over $A$ (see \cite[\S 7D]{MR632548}); thus
it induces a map $K_{1}(A) \rightarrow \zeta(A)^{\times}$, which we also denote by $\nr$.

Let $\mathcal C^{b} (\PMod (\mathfrak{A}))$ be the category of bounded complexes of finitely generated projective $\mathfrak{A}$-modules.
Then $K_{0}(\mathfrak{A}, A)$ identifies with the Grothendieck group whose generators are $[C^{\bullet}]$, where $C^{\bullet}$
is an object of the category $\mathcal C^{b}\tor(\PMod(\mathfrak{A}))$ of bounded complexes of finitely generated projective $\mathfrak{A}$-modules whose cohomology modules are $R$-torsion, and the relations are as follows: $[C^{\bullet}] = 0$ if $C^{\bullet}$ is acyclic, and
$[C_{2}^{\bullet}] = [C_{1}^{\bullet}] + [C_{3}^{\bullet}]$ for every short exact sequence
\begin{equation}\label{eq:SES-of-complexes}
0 \longrightarrow C_{1}^{\bullet} \longrightarrow C_{2}^{\bullet} \longrightarrow C_{3}^{\bullet} \longrightarrow 0
\end{equation}
in $\mathcal C^{b}\tor(\PMod(\mathfrak{A}))$ (see \cite[Chapter 2]{MR3076731} or \cite[\S 2]{zbMATH06148871}, for example).

Let $\mathcal{D} (\mathfrak{A})$ be the derived category of $\mathfrak{A}$-modules.
A complex of $\mathfrak{A}$-modules is said to be perfect if it is
isomorphic in $\mathcal{D} (\mathfrak{A})$ to an element of $\mathcal C^b(\PMod (\mathfrak{A}))$.
We denote the full triangulated subcategory of
$\mathcal{D} (\mathfrak{A})$ comprising perfect complexes by $\mathcal{D}^{\perf} (\mathfrak{A})$,
and the full triangulated subcategory
comprising perfect complexes whose cohomology modules are $R$-torsion by $\mathcal{D}^{\perf}\tor (\mathfrak{A})$.
Then any object of $\mathcal{D}^{\perf}\tor (\mathfrak{A})$ defines an element in
$K_{0}(\mathfrak{A}, A)$.
In particular, a finitely generated $R$-torsion $\mathfrak{A}$-module $M$ of finite projective dimension 
considered as a complex concentrated in degree $0$ defines an element $[M] \in K_{0}(\mathfrak{A}, A)$.

\subsection{Denominator ideals}\label{subsec:denom-ideals}
Let $R$ be a noetherian integrally closed domain with field of fractions $E$.
Let $A$ be a finite-dimensional separable $E$-algebra and let $\mathfrak{A}$ be an $R$-order in $A$.
We choose a maximal $R$-order $\mathfrak{M}$ such that $\mathfrak{A} \subseteq \mathfrak{M} \subseteq A$.
Following \cite[\S 3.6]{MR3092262}, for every matrix $H \in M_{b \times b} (\mathfrak{A})$ there is a generalised adjoint matrix
$H^{\ast} \in M_{b\times b}(\mathfrak{M})$ such that $H^{\ast} H = H H^{\ast} = \nr (H) \cdot 1_{b \times b}$
(note that the conventions in \cite[\S 3.6]{MR3092262} slightly differ from those in  \cite{MR2609173}).
If $\tilde{H} \in M_{b \times b} (\mathfrak{A})$ is a second matrix, then $(H \tilde{H})^{\ast} = \tilde{H}^{\ast} H^{\ast}$.
We define
\begin{eqnarray*}
\mathcal{H}(\mathfrak{A}) & := &
\{ x \in \zeta(\mathfrak{A}) \mid xH^{\ast} \in M_{n \times n}(\mathfrak{A}) \, \forall H \in M_{n \times n}(\mathfrak{A}) \, \forall n \in \N \},\\
\mathcal{I}(\mathfrak{A}) & := &
\langle \nr(H) \mid H \in  M_{n \times n}(\mathfrak{A}), \,  n \in \N\rangle_{\zeta(\mathfrak{A})}.
\end{eqnarray*}
One can show that these are $R$-lattices satisfying
\begin{equation}  \label{eqn:HI_equals_H}
\mathcal{H}(\mathfrak{A}) \cdot \mathcal I(\mathfrak{A}) = \mathcal{H}(\mathfrak{A}) 
\subseteq \zeta(\mathfrak{A}) \subseteq \mathcal{I}(\mathfrak{A}) \subseteq \zeta(\mathfrak{M}).
\end{equation}
Hence $\mathcal{H}(\mathfrak{A})$ is an ideal in the commutative $R$-order $\mathcal{I}(\mathfrak{A})$.
We will refer to $\mathcal{H}(\mathfrak{A})$ as the \emph{denominator ideal} of the $R$-order $\mathfrak{A}$.
If $p$ is a prime and $G$ is a finite group, we set
\begin{eqnarray*}
\mathcal{I}(G) := \mathcal{I}(\Z[G]), & &  \mathcal{I}_{p}(G) := \mathcal{I}(\Z_{p}[G]), \\
\mathcal{H}(G) := \mathcal{H}(\Z[G]), & &  \mathcal{H}_{p}(G) := \mathcal{H}(\Z_{p}[G]).
\end{eqnarray*}
The first claim of the following result is a special case of \cite[Proposition 4.4]{MR3092262}.
The second claim then follows easily from \eqref{eqn:HI_equals_H}.

\begin{prop}\label{prop:p-does-not-divide-order-of-comm-subgroup}
Let $p$ be prime and $G$ be a finite group. Then $\mathcal{H}_{p}(G) = \zeta(\Z_{p}[G])$ if and only if $p$ 
does not divide the order of the commutator subgroup of $G$. Moreover, in this case we have
$\mathcal{I}_{p}(G) = \zeta(\Z_{p}[G])$.
\end{prop}

The importance of the $\zeta(\mathfrak{A})$-module $\mathcal{H}(\mathfrak{A})$
comes from its relation to non-commutative Fitting invariants, which we introduce now.

\subsection{Non-commutative Fitting invariants}
For further details on the following material we refer the reader to \cite{MR2609173} and \cite{MR3092262}.
Let $A$ be a finite-dimensional separable algebra over a field $E$ and $\mathfrak{A}$ be an $R$-order in $A$, where $R$ is an integrally closed complete commutative noetherian local domain with field of fractions $E$.
For example, if $p$ is a prime we can take $\mathfrak{A}$ to be a $p$-adic group ring $\Z_{p}[G]$
where $G$ is a finite group or to be a completed group algebra $\Z_{p} \llbracket \mathcal{G} \rrbracket $ where $\mathcal{G}$
is a one-dimensional $p$-adic Lie group.

Let $X$ and $Y$ be two $\zeta(\mathfrak{A})$-submodules of an $R$-torsionfree $\zeta(\mathfrak{A})$-module.
Then $X$ and $Y$ are said to be \emph{$\nr(\mathfrak{A})$-equivalent} if there exists a positive integer $n$ and a matrix $U \in \GL_{n}(\mathfrak{A})$
such that $X = \nr(U) \cdot Y$.
We denote the corresponding equivalence class by $[X]_{\nr(\mathfrak{A})}$.
We say that $X$ is
$\nr(\mathfrak{A})$-contained in $Y$ (and write $[X]_{\nr(\mathfrak{A})} \subseteq [Y]_{\nr(\mathfrak{A})}$)
if for all $X' \in [X]_{\nr(\mathfrak{A})}$ there exists $Y' \in [Y]_{\nr(\mathfrak{A})}$
such that $X' \subseteq Y'$. Note that it suffices to check this property for one $X_{0} \in [X]_{\nr(\mathfrak{A})}$.
We will say that $x$ is contained in $[X]_{\nr(\mathfrak{A})}$ (and write $x \in [X]_{\nr(\mathfrak{A})}$) if there is $X_{0} \in [X]_{\nr(\mathfrak{A})}$ such that $x \in X_{0}$.

Further suppose that $X$ and $Y$ are in fact $\zeta(\mathfrak{A})$-submodules of $\zeta(A)$.
Let $X \cdot Y$ denote the $\zeta(\mathfrak{A})$-submodule of $\zeta(A)$ generated the set $\{ xy \mid x \in X, y \in Y \}$.
Then the product $[X]_{\nr(\mathfrak{A})} \cdot [Y]_{\nr(\mathfrak{A})} := [X \cdot Y]_{\nr(\mathfrak{A})}$ is well-defined. 
If $e \in A$ is a central idempotent then $e[X]_{\nr(\mathfrak{A})} := [eX]_{\nr(e\mathfrak{A})}$ is also well-defined. 
Moreover, if $X = \langle \alpha \rangle_{\zeta(\mathfrak{A})}$ is generated by
a single element $\alpha \in \zeta(A)^{\times}$, then we set
$[X]_{\nr(\mathfrak{A})}^{-1} := [\langle \alpha^{-1} \rangle_{\zeta(\mathfrak{A})}]_{\nr(\mathfrak{A})}$.

Now let $M$ be a (left) $\mathfrak{A}$-module with finite presentation
\begin{equation} \label{eqn:finite_presentation}
\mathfrak{A}^{a} \stackrel{h}{\longrightarrow} \mathfrak{A}^{b} \longrightarrow M \longrightarrow 0.
\end{equation}
We identify the homomorphism $h$ with the corresponding matrix in $M_{a \times b}(\mathfrak{A})$ and define
$S(h) = S_b(h)$ to be the set of all $b \times b$ submatrices of $h$ if $a \geq b$. In the case $a=b$
we call \eqref{eqn:finite_presentation} a \emph{quadratic presentation}.
The \emph{Fitting invariant} of $h$ over $\mathfrak{A}$ is defined to be
\[
\Fitt_{\mathfrak{A}}(h) = \left\{ \begin{array}{lll} [0]_{\nr(\mathfrak{A})} & \mbox{ if } & a<b \\
\left[\langle \nr(H) \mid H \in S(h)\rangle_{\zeta(\mathfrak{A})}\right]_{\nr(\mathfrak{A})} & \mbox{ if } & a \geq b. \end{array} \right.
\]
We call $\Fitt_{\mathfrak{A}}(h)$ a Fitting invariant of $M$ over $\mathfrak{A}$.
One defines $\Fitt_{\mathfrak{A}}^{\max}(M)$ to be the unique Fitting invariant of $M$ over $\mathfrak{A}$ which is maximal among all Fitting invariants of $M$ with respect to the partial order ``$\subseteq$''.
If $M$ admits a quadratic presentation $h$,
we set 
\begin{equation}\label{eq:quad-fitt-defn}
\Fitt_{\mathfrak{A}}(M) := \Fitt_{\mathfrak{A}}(h),
\end{equation}
which can be shown to be independent of the chosen quadratic presentation.

Now let $C^{\bullet} \in \mathcal{D}^{\perf}\tor(\mathfrak{A})$ and recall from \S \ref{subsec:K-theory}
that $C^{\bullet}$ defines an element $[C^{\bullet}]$ in the relative algebraic $K$-group $K_{0}(\mathfrak{A},A)$.
Recall the long exact sequence of $K$-theory \eqref{eqn:long-exact-seq}.
If $\rho([C^{\bullet}])=0$, 
we choose  $x \in K_{1}(A)$ such that $\partial(x) = [C^{\bullet}]$ and define
\begin{equation}\label{eqn:fitt-of-complex}
\Fitt_{\mathfrak{A}}(C^{\bullet}) := \left[\langle \nr(x) \rangle_{\zeta(\mathfrak{A})}\right]_{\nr(\mathfrak{A})}.
\end{equation}
Note that this is well-defined by the exactness of \eqref{eqn:long-exact-seq}.
Let $C^{\bullet}_{i} \in \mathcal{D}^{\perf}\tor(\mathfrak{A})$ such that $\rho([C_{i}^{\bullet}])=0$ for $i=1,2,3$.
Then if $[C_{2}^{\bullet}] = [C_{1}^{\bullet}] + [C_{3}^{\bullet}]$ in $K_{0}(\mathfrak{A},A)$
(this is the case in the situation of  \eqref{eq:SES-of-complexes}, for example)
it is straightforward to show that 
\begin{equation}\label{eqn:fitt-of-sum-of-complexes-in-rel-K-zero}
\Fitt_{\mathfrak{A}}(C_{2}^{\bullet}) = \Fitt_{\mathfrak{A}}(C_{1}^{\bullet}) \cdot \Fitt_{\mathfrak{A}}(C_{3}^{\bullet}).
\end{equation}
To put this in context, 
we note that if $C^{\bullet}$ is isomorphic in $\mathcal{D}(\mathfrak{A})$ to a complex $P^{-1} \rightarrow P^{0}$ concentrated in
degree $-1$ and $0$ such that $P^{-1}$ and $P^{0}$ are both finitely generated $R$-torsion $\mathfrak{A}$-modules
of projective dimension at most $1$, then
\begin{equation}\label{eq:rel-fitt-eq}
\Fitt_{\mathfrak{A}}(C^{\bullet}) = \Fitt_{\mathfrak{A}}(P^{0} : P^{-1}), 
\end{equation}
where the righthand side denotes the relative Fitting invariant of \cite[Definition 3.6]{MR2609173}.

\begin{remark}\label{rmk:quad-pres-rho-is-zero}
Let $M$ be a finitely generated $R$-torsion $\mathfrak{A}$-module of projective dimension at most $1$.
Then it is straightforward to show that $M$ admits a quadratic presentation if and only if $\rho([M])=0$
(see \cite[p.\ 2764]{MR2609173}). 
\end{remark}

\begin{remark}\label{rmk:quad-complex-defns-coincide}
Let $M$ be a finitely generated $R$-torsion $\mathfrak{A}$-module of projective dimension at most $1$ 
and assume that $M$ admits a quadratic presentation. 
Then one can consider $M$ as a complex concentrated in degree $0$ defining an element of 
$\mathcal{D}^{\perf}\tor(\mathfrak{A})$, and one can show that the definitions of $\Fitt_{\mathfrak{A}}(M)$
given by \eqref{eq:quad-fitt-defn} and \eqref{eqn:fitt-of-complex} coincide in this situation. 
\end{remark}

Non-commutative Fitting invariants provide a powerful tool for computing annihilators; for the following result
see \cite[Theorem 3.3]{MR3092262} or \cite[Theorem 4.2]{MR2609173}.

\begin{theorem}\label{thm:Fitt-annihilation}
If $M$ is a finitely presented $\mathfrak{A}$-module, then
\begin{equation}\label{eq:denom-times-fitting-in-annihilator} 
\mathcal{H}(\mathfrak{A}) \cdot \Fitt_{\mathfrak{A}}^{\max}(M) \subseteq \Ann_{\zeta(\mathfrak{A})}(M). 
\end{equation}
\end{theorem}

\begin{remark}
The inclusion \eqref{eq:denom-times-fitting-in-annihilator} should be interpreted as follows: 
if $x \in \zeta(\mathfrak{M})$ such that $x \in \Fitt_{\mathfrak{A}}^{\max}(M)$ (in the sense above)
and $h \in \mathcal{H}(\mathfrak{A})$ then $h \cdot x \in \Ann_{\zeta(\mathfrak{A})}(M)$.
\end{remark}

We list some properties of non-commutative Fitting invariants which we will use later.

\begin{lemma}\label{lem:Fitting-properties}
Let $M$ and $M'$ be finitely presented $\mathfrak{A}$-modules and let $e \in A$ be a central idempotent.
Then the following statements hold.
\begin{enumerate}
\item
If $M \twoheadrightarrow M'$ is a surjection then $\Fitt_{\mathfrak{A}}^{\max}(M) \subseteq \Fitt_{\mathfrak{A}}^{\max}(M')$.
\item 
If $M$ and $M'$ admit quadratic presentations then so does $M \oplus M'$ and we have an equality
$
\Fitt_{\mathfrak{A}}(M) \cdot \Fitt_{\mathfrak{A}}(M') =  \Fitt_{\mathfrak{A}}(M \oplus M')
$.
\item
We have an inclusion $e \Fitt_{\mathfrak{A}}^{\max}(M) \subseteq \Fitt_{e \mathfrak{A}}^{\max}(e \mathfrak{A} \otimes_{\mathfrak{A}} M)$
with equality if $e \in \mathfrak{A}$.
\end{enumerate}
\end{lemma}

\begin{proof}
For (i) see \cite[Theorem 3.1 (i)]{MR3092262}.
Part (ii) is a special case of \cite[Theorem 3.1 (iii)]{MR3092262}.
The first claim of part (iii) is \cite[Theorem 3.1 (vi)]{MR3092262}, and the second claim follows easily from the definition of Fitting invariants
and the decomposition $\mathfrak{A} = e \mathfrak{A}  \oplus (1-e) \mathfrak{A} $.
\end{proof}

\begin{lemma}\label{lem:fitt-eq-complex-quad}
Let $A$ and $B$ be finitely generated $R$-torsion $\mathfrak{A}$-modules 
of projective dimension at most $1$ and with quadratic presentations.
Let $A \rightarrow B$ be a complex concentrated in degrees $-1$ and $0$. 
Then recalling Remark \ref{rmk:quad-complex-defns-coincide} we have
\[
\Fitt_{\mathfrak{A}}(B:A) = \Fitt_{\mathfrak{A}}(A \rightarrow B) = \Fitt_{\mathfrak{A}}^{-1}(A) \cdot \Fitt_{\mathfrak{A}}(B).
\] 
\end{lemma}

\begin{proof}
The first equality follows from \eqref{eq:rel-fitt-eq}.
We consider $A$ and $B$ as complexes concentrated in degree $0$.
Then we have a short exact sequence of complexes
\[
0 \longrightarrow B \longrightarrow (A \rightarrow B) \longrightarrow A[1] \longrightarrow 0,
\] 
where $A[1]$ is concentrated in degree $-1$. 
Hence by \eqref{eq:SES-of-complexes} we have 
\[
[A \rightarrow B] = [B] + [A[1]]  = [B] - [A],
\]
in $K_{0}(\mathfrak{A},A)$ and so the desired result now follows from \eqref{eqn:fitt-of-sum-of-complexes-in-rel-K-zero}.  
\end{proof}

\section{The non-abelian Brumer--Stark conjecture} \label{sec:Brumer--Stark}

\subsection{Ray class groups}\label{subsec:ray-class-groups}
Let $L/K$ be a finite Galois extension of number fields with Galois group $G$.
For each place $v$ of $K$ we fix a place $w$ of $L$ above $v$ and write $G_{w}$ and $I_{w}$ for the decomposition
group and inertia subgroup of $L/K$ at $w$, respectively.
When $w$ is a finite place, we  choose a lift $\phi_{w} \in G_{w}$ of the Frobenius automorphism at $w$;
moreover, we write $\mathfrak{P}_{w}$ for the associated prime ideal in $L$ and $\ord_{w}$ for the associated valuation.

For any set $S$ of places of $K$, we write $S(L)$ for the set of places of $L$ which lie above those in $S$.
Now let $S$ be a finite set of places of $K$ containing the set $S_{\infty}=S_{\infty}(K)$ of archimedean places
and let $T$ be a second finite set of places of $K$ such that $S \cap T = \emptyset$.
We write $\cl_{L}^{T}$ for the ray class group of $L$ 
associated to the modulus
$\mathfrak{M}_{L}^{T} := \prod_{w \in T(L)} \mathfrak{P}_w$ and $\mathcal{O}_{L,S}$ for the ring of $S(L)$-integers in $L$.
Let $\mathcal{O}_{L} := \mathcal{O}_{L, S_{\infty}}$ be the ring of integers in $L$.
Let $S_{f}$ be the set of all finite primes in $S$;
then there is a natural map $\Z S_{f}(L) \to \cl_{L}^{T}$ which sends each place $w \in S_{f}(L)$
to the corresponding class $[\mathfrak{P}_w] \in \cl_{L}^{T}$. We denote the cokernel of this map by $\cl_{L,S}^{T}$.
Moreover, we denote the $S(L)$-units of $L$ by $E_{L,S}$ and define
$E_{L,S}^{T} := \left\{x \in E_{L,S}: x \equiv 1 \bmod \mathfrak{M}_{L}^{T} \right\}$.
All these modules are equipped with a natural $G$-action and we have the following exact sequences of $\Z[G]$-modules.
If $\Sigma$ is a subset of $S$ containing $S_{\infty}$, then we have
\begin{equation}\label{eqn:ray_class_sequence_ZS}
0 \longrightarrow E_{L, \Sigma}^T \longrightarrow E_{L,S}^T \stackrel{v_{L}}{\longrightarrow}
\Z [S(L) - \Sigma(L)] \longrightarrow \cl_{L, \Sigma}^{T} \longrightarrow \cl_{L,S}^{T} \longrightarrow 0,
\end{equation}
where $v_{L}(x) := \sum_{w \in S(L) - \Sigma(L)} \ord_w(x) w$ for every $x \in E_{L,S}^T$, and
\begin{equation}\label{eqn:ray_class_sequence}
0 \longrightarrow E_{L,S}^T \longrightarrow E_{L,S} \longrightarrow (\mathcal{O}_{L,S} / \mathfrak{M}_{L}^{T})\mal
\stackrel{\nu}{\longrightarrow} \cl_{L,S}^{T} \longrightarrow \cl_{L,S} \longrightarrow 0,
\end{equation}
where the map $\nu$ lifts an element $\overline x \in (\mathcal{O}_{L,S} / \mathfrak{M}_{L}^{T})^{\times}$ to
$x \in \mathcal{O}_{L,S}$ and
sends it to the ideal class $[(x)] \in \cl_{L,S}^{T}$ of the principal ideal $(x)$.

\subsection{Equivariant Artin $L$-values}\label{subsec:L-values}

Let $S$ be a finite set of places of $K$ containing $S_{\infty}$.
Let $\Irr_{\C}(G)$ denote the set of complex irreducible characters of $G$.
For $\chi \in \Irr_{\C}(G)$, we write $L_{S}(s,\chi)$ for the $S$-truncated Artin $L$-function attached to $\chi$ and $S$
(see  \cite[Chapter 0, \S 4]{MR782485}).
Recall that there is a canonical isomorphism
$\zeta(\C[G]) \simeq \prod_{\chi \in \Irr_{\C} (G)} \C$.
We define the equivariant $S$-truncated Artin $L$-function to be the meromorphic $\zeta(\C[G])$-valued function
\[
L_{S}(s) := (L_{S}(s,\chi))_{\chi \in \Irr_{\C} (G)}.
\]
For $\chi \in \Irr_{\C}(G)$, let $V_{\chi}$ be a left $\C[G]$-module with character $\chi$.
If $T$ is a second finite set of places of $K$ such that $S \cap T = \emptyset$, we define
\[
\delta_{T}(s,\chi) = \prod_{v \in T} \det(1 - N(v)^{1-s} \phi_{w}^{-1} \mid V_{\chi}^{I_{w}}) \quad  \textrm{ and }  \quad
\delta_{T}(s) := (\delta_{T}(s,\chi))_{\chi\in \Irr_{\C} (G)}.
\]
We set
\[
\Theta_{S,T}(s) := \delta_{T}(s) \cdot L_{S}(s)^{\sharp},
\]
where $^{\sharp}: \C[G] \to \C[G]$ denotes the anti-involution induced by $g \mapsto g^{-1}$ for $g \in G$.
Note that $L_{S}(s)^{\sharp} = (L_{S}(s,\check{\chi}))_{\chi \in \Irr_{\C} (G)}$ 
where $\check \chi$ denotes the character contragredient to $\chi$.
The functions $\Theta_{S,T}(s)$ are the so-called $(S,T)$-modified $G$-equivariant $L$-functions and we define Stickelberger elements
\[
\theta_{S}^{T}(L/K) = \theta_{S}^{T} := \Theta_{S,T}(0) \in \zeta(\Q[G]).
\]
Note that a priori we only have $\theta_{S}^{T} \in \zeta(\C[G])$, but by a result of Siegel \cite{MR0285488} we know that $\theta_{S}^{T}$
in fact belongs to $\zeta(\Q[G])$.
If $T$ is empty, we abbreviate $\theta_{S}^{T}$ to $\theta_{S}$.

Let $p$ be a prime and let $\iota : \C_{p} \rightarrow \C$ be a field isomorphism.
Then the image of $\theta_{S}^{T}$ under the canonical maps 
\begin{equation}\label{eqn:embeddings-Q-Qp-Cp}
\zeta(\Q[G]) \hookrightarrow \zeta(\Q_{p}[G]) \hookrightarrow \zeta(\C_{p}[G]) \cong \textstyle{\prod_{\chi \in \Irr_{\C_{p}}(G)}} \C_{p} 
\end{equation}
is given by 
$(\iota^{-1}(\delta_{T}(0, \iota \circ \chi) L_{S}(0,\iota \circ \check{\chi})))_{\chi \in \Irr_{\C_{p}}(G)}$ 
and this is independent of the choice of $\iota$.
We shall henceforth consider $\theta_{S}^{T}$ as an element of $\zeta(\Q_{p}[G])$ or $\zeta(\C_{p}[G])$ via 
\eqref{eqn:embeddings-Q-Qp-Cp} when convenient.
Moreover, we shall often drop $\iota$ and $\iota^{-1}$ from the notation.

\subsection{Reduction to CM-extensions}\label{subsec:reduce-to-CM}
Let $\chi \in \Irr_{\C}(G)$. The order of vanishing formula for $L_{S}(s,\chi)$ at $s=0$
(see \cite[Chapter I, Proposition 3.4]{MR782485}) shows that if either 
$\chi$ is non-trivial and $S$ contains an (infinite) place $v$ such that $V_{\chi}^{G_{w}} \neq 0$
or $\chi$ is trivial and $|S| > 1$ then the $\chi$-part of $\theta^{T}_{S}$ vanishes.
Hence if  $\theta_{S}^{T}$ is non-trivial, precisely one of the following possibilities occurs:
(i) $K$ is totally real and $L$ is totally complex,
(ii) $K$ is an imaginary quadratic field,  $L/K$ is unramified and $S = S_{\infty}$ or 
(iii) $L=K=\Q$ and $S = S_{\infty}$. 
In case (iii), the Brumer--Stark conjecture is trivial. 
For case (ii), see \cite[Remark 6.3]{MR3383600} for $G$ abelian; 
the situation in which $G$ is non-abelian has been considered recently
by Nomura \cite{MR3766914}. 
Finally, case (i) can often be reduced to the case that $L$ is a CM-field 
(see \cite[Proposition 6.4]{MR3383600} for $G$ abelian; the same argument works for general $G$ under the assumptions of \cite[Proposition 4.1]{MR3552493} with $r=0$). 
Therefore, we shall henceforth assume that  $L/K$ is a CM-extension, 
that is, $L$ is a CM-field, $K$ is totally real and complex conjugation induces a unique automorphism
$j$ of $L$ lying in the centre of $G$.

\subsection{The non-abelian Brumer and Brumer--Stark conjectures}\label{subsec:non-abelian-Brumer}
Assume that $L/K$ is a CM-extension and 
let $S_{\ram}=S_{\ram}(L/K)$ be the set of all places of $K$ that ramify in $L/K$.

\begin{hypothesis*}
Let $S$ and $T$ be finite sets of places of $K$. 
We say that $\Hyp(S,T)$ is satisfied if
(i)
$S_{\ram} \cup S_{\infty} \subseteq S$,
(ii)
$S \cap T = \emptyset$, and
(iii)
$E_{L,S}^T$ is torsionfree.
\end{hypothesis*}

\begin{remark}\label{rmk:conditions-on-T}
Condition (iii) means that there are no roots of unity of $L$ congruent to $1$ modulo all primes in $T(L)$. 
In particular, this will be satisfied if $T$ contains primes of two different residue characteristics or at least one prime of sufficiently large norm. 
\end{remark}

We choose a maximal order $\mathfrak{M}(G)$ such that $\Z[G] \subseteq \mathfrak{M}(G) \subseteq \Q[G]$.
For a fixed choice of $S$ we define $\mathfrak{A}_{S}$ to be the $\zeta(\Z[G])$-submodule of $\zeta(\mathfrak{M}(G))$ generated
by the  elements $\delta_T(0)$, where $T$ runs through the finite sets of places of $K$
such that $\Hyp(S \cup S_{\ram} \cup S_{\infty},T)$ is satisfied.
The following conjecture was formulated in \cite{MR2976321} and is a non-abelian generalisation of Brumer's conjecture.

\begin{conj}[$B(L/K,S)$] \label{conj:Brumer}
Let $S$ be a finite set of places of $K$ containing $S_{\ram} \cup S_{\infty}$.
Then $\mathfrak{A}_S \theta_S \subseteq \mathcal{I}(G)$ and for each $x \in \mathcal H(G)$ we have
\[
x \cdot \mathfrak{A}_S \theta_S \subseteq \Ann_{\Z[G]} (\cl_L).
\]
\end{conj}

\begin{remark}
If $G$ is abelian, \cite[Lemma 1.1, p.~82]{MR782485} implies that $\mathfrak{A}_{S} = \Ann_{\Z[G]}(\mu_{L})$.
In this case the results in \cite{MR525346, MR524276, MR579702} each imply that
$\mathfrak{A}_{S} \theta_{S} \subseteq \mathcal{I}(G) = \Z[G]$ and,
since $\mathcal{H}(G) = \Z[G]$ in this case, Conjecture \ref{conj:Brumer} recovers Brumer's conjecture.
\end{remark}

\begin{remark}\label{rmk:p-part-of-Brumer}
If $M$ is a finitely generated $\Z$-module and $p$ is a prime, we define its $p$-part to be $M(p) := \Z_{p} \otimes_{\Z} M$.
Replacing the class group $\cl_{L}$ by $\cl_{L}(p)$ for each prime $p$,
Conjecture $B(L/K,S)$ naturally decomposes into local conjectures $B(L/K,S,p)$.
It is then possible to replace $\mathcal{H}(G)$ by $\mathcal{H}_{p}(G)$ by \cite[Lemma 1.4]{MR2976321}.
Moreover, if $p$ does not divide the order of the commutator subgroup of $G$ then 
$\mathcal{H}_{p}(G) = \mathcal{I}_{p}(G) = \zeta(\Z_{p}[G])$ 
by Proposition \ref{prop:p-does-not-divide-order-of-comm-subgroup} and so after 
the hypotheses on $S$ the statement of the local conjecture simplifies to
\[
\mathfrak{A}_S \theta_S \subseteq \Ann_{\zeta(\Z_{p}[G])} (\cl_{L}(p)).
\]
\end{remark}

\begin{remark}
Burns \cite{MR2845620} has also formulated a conjecture which generalises many refined Stark conjectures to the
non-abelian situation. In particular, it implies Conjecture \ref{conj:Brumer} (see \cite[Proposition 3.5.1]{MR2845620}).
\end{remark}

For $\alpha \in L^{\times}$ we define
\[
S_{\alpha} :=  \{ v \mbox{ finite place of } K \mid \ord_{v}(N_{L/K}(\alpha))>0 \}.
\]
We call $\alpha$ an {\it anti-unit} if $\alpha^{1+j} = 1$.
Let $\omega_L := \nr (|\mu_L|)$. The following is a non-abelian generalisation of the Brumer--Stark conjecture (\cite[Conjecture 2.7]{MR2976321}).

\begin{conj}[$BS(L/K,S)$] \label{conj:Brumer--Stark}
Let $S$ be a finite set of places of $K$ containing $S_{\ram} \cup S_{\infty}$.
Then $\omega_L \cdot \theta_S \in \mathcal I(G)$ and for each $x \in \mathcal H(G)$ and each fractional ideal $\mathfrak{a}$ of $L$,
there is an anti-unit $\alpha = \alpha(x,\mathfrak{a},S) \in L\mal$ such that
\[
\mathfrak{a}^{x \cdot \omega_L \cdot \theta_S} = (\alpha)
\]
and for each finite set $T$ of primes of $K$ such that $\Hyp(S \cup S_{\alpha},T)$ is satisfied there is an $\alpha_{T} \in E_{L,S_{\alpha}}^{T}$
such that
\begin{equation} \label{eqn:abelian-ersatz}
\alpha^{z \cdot \delta_T(0)} = \alpha_T^{z \cdot \omega_L}
\end{equation}
for each $z \in \mathcal H(G)$.
\end{conj}

\begin{remark}
If $G$ is abelian, we have $\mathcal{I}(G) = \mathcal{H}(G) = \Z[G]$ and $\omega_{L} = |\mu_{L}|$.
Hence it suffices to treat the case $x=z=1$ in this situation.
Then \cite[Proposition 1.2, p.~83]{MR782485} states that condition \eqref{eqn:abelian-ersatz}
on the anti-unit $\alpha$ is equivalent to the assertion that the extension $L(\alpha^{1/\omega_{L}}) / K$ is abelian.
\end{remark}

\begin{remark}
As in Remark \ref{rmk:p-part-of-Brumer}, we obtain local conjectures $BS(L/K,S,p)$ for each prime $p$.
Again, these local conjectures simplify to a version without denominator ideals in the case the $p$ does not divide the order of the commutator subgroup of $G$.
\end{remark}

\subsection{A criterion involving Pontryagin duals and Fitting invariants}\label{subsec:criterion-pont-fitt}
For an abstract abelian group $A$ we write $A^{\vee}$ for $\Hom(A, \Q / \Z)$. This induces an equivalence between
the categories of abelian profinite groups and discrete abelian torsion groups (see \cite[Theorem 1.1.11]{MR2392026} and the discussion thereafter). 
For a finitely generated $\Z_{p}[G]$-module $M$, we have $M^{\vee} = \Hom_{\Z_{p}}(M, \Q_{p} / \Z_{p})$, 
and this is endowed with the contragredient $G$-action $(gf)(m) = f (g^{-1} m)$ for $f \in M^{\vee}$, $g \in G$ and $m \in M$.

For a $G$-module $M$ we write $M^{+}$ and $M^{-}$ for the submodules of $M$ upon which $j$ acts as $1$ and $-1$, respectively.
In particular, we shall be interested in $(\cl_{L,S}^{T}(p))^-$ for odd primes $p$; 
we will abbreviate this module to $A_{L,S}^{T}$ when $p$ is clear from context. 
Note that $A_{L,S}^{T}$ is a finite module over the ring $\Z_{p}[G]_{-} := \Z_{p}[G]/(1+j)$.
We shall need the following variant of \cite[Proposition 3.9]{MR2976321}.

\begin{prop}\label{prop:SBS_implies_BS}
Let $S$ be a finite set of places of $K$ containing $S_{\ram} \cup S_{\infty}$ and let $p$ be an odd prime.
Suppose that for every finite set $T$ of places of $K$ such that $\Hyp(S,T)$ is satisfied we have
\begin{equation}\label{eqn:dual-to-strong-Brumer--Stark}
(\theta_S^T)^{\sharp} \in \Fitt^{\max}_{\Z_{p}[G]_{-}}((A_{L}^{T})^{\vee}). 
\end{equation}
Then both $BS(L/K,S,p)$ and $B(L/K,S,p)$ are true.
\end{prop}

\begin{remark}
The containment \eqref{eqn:dual-to-strong-Brumer--Stark} may be considered as a `dual version'
of the so-called strong Brumer--Stark property, which is fulfilled if $\theta_{S}^{T} \in \Fitt_{\Z_{p}[G]_{-}}^{\max}(A_{L}^{T})$ (see \cite[Definition 3.6]{MR2976321}). 
\end{remark}

\begin{proof}[Proof of Proposition \ref{prop:SBS_implies_BS}]
This has already been shown within the proof of \cite[Corollary 4.6]{MR3072281}, but we repeat the argument here for the convenience of the reader.
Since $BS(L/K,S,p)$ implies $B(L/K,S,p)$ by \cite[Lemma 2.12]{MR2976321}, we need only treat the case of the Brumer--Stark conjecture.
Moreover, \cite[Proposition 3.9]{MR2976321} says that $BS(L/K,S,p)$ is implied by the strong Brumer--Stark property.
However, the proof of \cite[Proposition 3.9]{MR2976321} carries over unchanged once we observe that
\[
\Ann_{\Z_{p}[G]_{-}}(M) = \Ann_{\Z_{p}[G]_{-}}(M^{\vee})^{\sharp}
\]
for every finite $\Z_{p}[G]_{-}$-module $M$.
\end{proof}

\section{The equivariant Iwasawa main conjecture}\label{sec:EIMC}

\subsection{Certain one-dimensional $p$-adic Lie extensions}\label{subsec:certain-one-dim-Lie}
Let $p$ be an odd prime and let $K$ be a number field. 
Let $\mathcal{L}/K$ be a Galois extension such that $\mathcal{L}$ contains the cyclotomic 
$\Z_{p}$-extension $K_{\infty}$ of $K$ and $[\mathcal{L} : K_{\infty}]$ is finite. 
Then the Galois group $\mathcal{G} := \Gal(\mathcal{L}/K)$ is a one-dimensional $p$-adic Lie group.
Let $H=\Gal(\mathcal{L}/K_{\infty})$ and let $\Gamma_{K}=\Gal(K_{\infty}/K)$.
Let $\gamma_{K}$ be a topological generator of $\Gamma_{K} \simeq \Z_{p}$.
The argument given in \cite[\S 1]{MR2114937} shows there exists a lift $\gamma \in \mathcal{G}$ of $\gamma_{K}$
that induces a splitting of the short exact sequence
\[
1 \longrightarrow H \longrightarrow \mathcal{G} \longrightarrow \Gamma_{K} \longrightarrow 1.
\]
Thus we obtain a semidirect product $\mathcal{G} = H \rtimes \Gamma$ where $\Gamma \simeq \Z_{p}$ 
is the pro-cyclic subgroup of $\mathcal{G}$ topologically generated by $\gamma$.
Since any homomorphism $\Gamma \rightarrow \Aut(H)$ must have open kernel, 
we may choose a natural number $n$ such that $\gamma^{p^n}$ is central in $\mathcal{G}$.
We fix such an $n$ and set $\Gamma_{0} := \Gamma^{p^n}$; hence $\Gamma_{0} \simeq \Z_{p}$ is contained in the centre of $\mathcal{G}$.

\subsection{The Iwasawa algebra as an order}
The Iwasawa algebra of $\mathcal{G}$ is
\[
\Lambda(\mathcal{G}) := \Z_{p} \llbracket \mathcal{G} \rrbracket  = \varprojlim \Z_{p}[\mathcal{G}/\mathcal{N}],
\]
where the inverse limit is taken over all open normal subgroups $\mathcal{N}$ of $\mathcal{G}$.
If $F$ is a finite field extension of $\Q_{p}$  with ring of integers $\mathcal{O}=\mathcal{O}_{F}$,
we put $\Lambda^{\mathcal{O}}(\mathcal{G}) := \mathcal{O} \otimes_{\Z_{p}} \Lambda(\mathcal{G}) = \mathcal{O} \llbracket \mathcal{G} \rrbracket $.
There is a ring isomorphism
$R:=\mathcal{O} \llbracket \Gamma_{0} \rrbracket  \simeq \mathcal{O} \llbracket T \rrbracket $ induced by $\gamma^{p^n} \mapsto 1+T$
where $\mathcal{O} \llbracket T \rrbracket $ denotes the power series ring in one variable over $\mathcal{O}$.
If we view $\Lambda^{\mathcal{O}}(\mathcal{G})$ as an $R$-module (or indeed as a left $R[H]$-module), there is a decomposition
\begin{equation}\label{eq:Lambda-R-decomp}
\Lambda^{\mathcal{O}}(\mathcal{G}) = \bigoplus_{i=0}^{p^n-1} R[H] \gamma^{i}.
\end{equation}
Hence $\Lambda^{\mathcal{O}}(\mathcal{G})$ is finitely generated as an $R$-module and is an $R$-order in the separable $E:=Quot(R)$-algebra
$\mathcal{Q}^{F} (\mathcal{G})$, the total ring of fractions of $\Lambda^{\mathcal{O}}(\mathcal{G})$, obtained
from $\Lambda^{\mathcal{O}}(\mathcal{G})$ by adjoining inverses of all central regular elements.
Note that $\mathcal{Q}^{F} (\mathcal{G}) =  E \otimes_{R} \Lambda^{\mathcal{O}}(\mathcal{G})$ and that by
\cite[Lemma 1]{MR2114937} we have $\mathcal{Q}^{F} (\mathcal{G}) = F \otimes_{\Q_{p}} \mathcal{Q}(\mathcal{G})$,
where $\mathcal{Q}(\mathcal{G}) := \mathcal{Q}^{\Q_{p}}(\mathcal{G})$.

\subsection{An exact sequence of algebraic $K$-groups}\label{subsec:exact-seq-K-groups}
Specialising \eqref{eqn:long-exact-seq} to the situation $A = \mathcal{Q}(\mathcal{G})$ and $\mathfrak{A} = \Lambda(\mathcal{G})$
and applying \cite[Corollary 3.8]{MR3034286} (see \cite[\S 4.1]{MR3749195} for further explanation) we obtain the exact sequence
\begin{equation}\label{eqn:Iwasawa-K-sequence}
K_{1}(\Lambda(\mathcal{G})) \longrightarrow K_{1}(\mathcal{Q}(\mathcal{G})) \stackrel{\partial}{\longrightarrow}
K_{0}(\Lambda(\mathcal{G}),\mathcal{Q}(\mathcal{G})) \stackrel{\rho}\longrightarrow 0.
\end{equation}

\begin{remark}\label{rmk:M-admits-quadratic-presentation}
If $M$ is a finitely generated $R$-torsion $\Lambda(\mathcal{G})$-module of projective dimension at most $1$,
then combining the triviality of $\rho$ in \eqref{eqn:Iwasawa-K-sequence} with Remark \ref{rmk:quad-pres-rho-is-zero} shows that $M$ admits a quadratic presentation. 
\end{remark}

\subsection{Characters and central primitive idempotents} \label{subsec:idempotents}
Fix a character $\chi \in \Irr_{\Q_{p}^{c}}(\mathcal{G})$ 
(i.e.\ an irreducible $\Q_{p}^{c}$-valued character of $\mathcal{G}$ with open kernel)
and let $\eta$ be an irreducible constituent of
$\res^{\mathcal{G}}_{H} \chi$.
Then $\mathcal{G}$ acts on $\eta$ as $\eta^{g}(h) = \eta(g^{-1}hg)$
for $g \in \mathcal{G}$, $h \in H$, and following \cite[\S 2]{MR2114937} we set
\[
St(\eta) := \{g \in \mathcal{G}: \eta^g = \eta \}, \quad e(\eta) := \frac{\eta(1)}{|H|} \sum_{h \in H} \eta(h^{-1}) h,
\quad e_{\chi} := \sum_{\eta \mid \res^{\mathcal{G}}_{H} \chi} e(\eta).
\]
By \cite[Corollary to Proposition 6]{MR2114937} $e_{\chi}$ is a primitive central idempotent of
$\mathcal{Q}^{c}(\mathcal{G}) := \Q_{p}^{c} \otimes_{\Q_{p}} \mathcal{Q}(\mathcal{G})$.
In fact, every primitive central idempotent of $\mathcal{Q}^{c}(\mathcal{G})$ is of this form
and $e_{\chi} = e_{\chi'}$ if and only if $\chi = \chi' \otimes \rho$ for some character $\rho$ of $\mathcal{G}$ of type $W$
(i.e.~$\res^{\mathcal{G}}_{H} \rho = 1$).
Let $w_{\chi} = [\mathcal{G} : St(\eta)]$ and note that this is a power of $p$ since $H$ is a subgroup of $St(\eta)$.

Let $F/\Q_{p}$ be a finite extension over which both characters $\chi$ and $\eta$ have realisations.
Let $V_{\chi}$ denote a realisation of $\chi$ over $F$.	
By \cite[Propositions 5 and 6]{MR2114937}, there exists a unique element $\gamma_{\chi} \in \zeta(\mathcal{Q}^{F}(\mathcal{G})e_{\chi})$ 
such that $\gamma_{\chi}$ acts trivially on $V_{\chi}$ and $\gamma_{\chi} = g_{\chi}c_{\chi}$ with $g_{\chi} \in \mathcal{G}$ mapping to
$\gamma_{K}^{w_{\chi}} \bmod H$ and with $c_{\chi} \in (F[H]e_{\chi})^{\times}$.
Moreover, $\gamma_{\chi}$ generates a pro-cyclic $p$-subgroup $\Gamma_{\chi}$ of $\mathcal{Q}^{F}(\mathcal{G})e_{\chi}$ and induces an isomorphism $\mathcal{Q}^{F}(\Gamma_{\chi}) \stackrel{\simeq}{\longrightarrow} \zeta(\mathcal{Q}^{F} (\mathcal{G})e_{\chi})$.

\subsection{Determinants and reduced norms}\label{subsec:dets-and-nr}
Following \cite[Proposition 6]{MR2114937}, we define a map
\[
j_{\chi}: \zeta(\mathcal{Q}^{F} (\mathcal{G})) \twoheadrightarrow \zeta(\mathcal{Q}^{F} (\mathcal{G})e_{\chi}) \simeq \mathcal{Q}^{F}(\Gamma_{\chi}) \rightarrow  \mathcal{Q}^{F}(\Gamma_{K}),
\]
where the last arrow is induced by mapping $\gamma_{\chi}$ to $\gamma_{K}^{w_{\chi}}$.
It follows from op.\ cit.\ that $j_{\chi}$ is independent of the choice of $\gamma_{K}$ and that 
for every matrix $\Theta \in M_{n \times n} (\mathcal{Q}(\mathcal{G}))$ we have
\begin{equation} \label{eqn:jchi-det}
j_{\chi} (\nr(\Theta)) = \mathrm{det}_{\mathcal{Q}^{F}(\Gamma_{K})} (\Theta \mid \Hom_{F[H]}(V_{\chi},  \mathcal{Q}^{F}(\mathcal{G})^n)).
\end{equation}
Here, $\Theta$ acts on $f \in \Hom_{F[H]}(V_{\chi},  \mathcal{Q}^{F}(\mathcal{G})^{n})$ via right multiplication,
and $\gamma_{K}$ acts on the left via $(\gamma_{K} f)(v) = \gamma \cdot f(\gamma^{-1} v)$ for all $v \in V_{\chi}$
which is easily seen to be independent of the choice of $\gamma$.
Hence the map
\begin{eqnarray*}
\Det(~)(\chi): K_{1}(\mathcal{Q}(\mathcal{G})) & \rightarrow & \mathcal{Q}^{F}(\Gamma_{K})^{\times} \\
 {[P,\alpha]}& \mapsto & \mathrm{det}_{\mathcal{Q}^{F}(\Gamma_{K})} (\alpha \mid \Hom_{F[H]}(V_{\chi},  F \otimes_{\Q_{p}} P)),
\end{eqnarray*}
where $P$ is a projective $\mathcal{Q}(\mathcal{G})$-module and $\alpha$ a $\mathcal{Q}(\mathcal{G})$-automorphism of $P$, is just $j_{\chi} \circ \nr$ (see \cite[\S 3, p.\ 558]{MR2114937}).
If $\rho$ is a character of $\mathcal{G}$ of type $W$ (i.e.~$\res^{\mathcal{G}}_H \rho = 1$)
then we denote by
$\rho^{\sharp}$ the automorphism of the field $\mathcal{Q}^{c}(\Gamma_{K})$ induced by
$\rho^{\sharp}(\gamma_{K}) = \rho(\gamma_{K}) \gamma_{K}$. 
Moreover, we denote the additive group generated by all $\Q_{p}^{c}$-valued
characters of $\mathcal{G}$ with open kernel by $R_p(\mathcal{G})$; finally, 
$\Hom^{\ast}(R_{p}( \mathcal{G}), \mathcal{Q}^{c}(\Gamma_{K})^{\times})$
is the group of all homomorphisms 
$f: R_p(\mathcal{G}) \rightarrow \mathcal{Q}^{c}(\Gamma_{K})\mal$ satisfying
\[
\begin{array}{ll}
f(\chi \otimes \rho) = \rho^{\sharp}(f(\chi)) & \mbox{ for all characters } \rho \mbox{ of type } W \mbox{ and}\\
f({}^{\sigma}\chi) = \sigma(f(\chi)) & \mbox{ for all Galois automorphisms } \sigma \in \Gal(\Q_{p}^{c}/\Q_{p}).
\end{array}
\]
By \cite[Proof of Theorem 8]{MR2114937} we have an isomorphism
\begin{eqnarray*}
\zeta(\mathcal{Q}(\mathcal{G}))\mal & \simeq & 
\Hom^{\ast}(R_{p}(\mathcal{G}), \mathcal{Q}^{c}(\Gamma_{K})^{\times})\\
x & \mapsto & [\chi \mapsto j_{\chi}(x)].
\end{eqnarray*}
By \cite[Theorem 8]{MR2114937} the map $\Theta \mapsto [\chi \mapsto \Det(\Theta)(\chi)]$
defines a homomorphism
\[
\Det: K_{1}(\mathcal{Q}(\mathcal{G})) \longrightarrow \Hom^{\ast}(R_p(\mathcal{G}), \mathcal{Q}^{c}(\Gamma_{K})\mal)
\]
such that we obtain a commutative triangle
\begin{equation} \label{eqn:Det_triangle}
\xymatrix{
& K_{1}(\mathcal{Q}(\mathcal{G})) \ar[dl]_{\nr} \ar[dr]^{\Det} &\\
{\zeta(\mathcal{Q}(\mathcal{G}))^{\times}} \ar[rr]^{\sim} & & {\Hom^{\ast}(R_p( \mathcal{G}), \mathcal{Q}^{c}(\Gamma_{K})^{\times})}.}
\end{equation}

\subsection{The $p$-adic cyclotomic character and its projections}\label{subsec:cyclotomic-char}
Let $\chi_{\mathrm{cyc}}$ be the $p$-adic cyclotomic character
\[
\chi_{\mathrm{cyc}}: \Gal(\mathcal{L}(\zeta_{p})/K) \longrightarrow \Z_{p}^{\times},
\]
defined by $\sigma(\zeta) = \zeta^{\chi_{\mathrm{cyc}}(\sigma)}$ for any $\sigma \in \Gal(\mathcal{L}(\zeta_{p})/K)$ and any $p$-power root of unity $\zeta$.
Let $\omega$ and $\kappa$ denote the composition of $\chi_{\mathrm{cyc}}$ with the projections onto the first and second factors of the canonical decomposition $\Z_{p}^{\times} = \mu_{p-1} \times (1+p\Z_{p})$, respectively;
thus $\omega$ is the Teichm\"{u}ller character.
We note that $\kappa$ factors through $\Gamma_{K}$ 
(and thus also through $\mathcal{G}$) and by abuse of notation we also 
use $\kappa$ to denote the associated maps with these domains.
We put $u := \kappa(\gamma_{K})$.
For $r \in \N_{0}$ divisible by $p-1$ 
(or more generally divisible by the degree $[\mathcal{L}(\zeta_{p}) : \mathcal{L}]$), 
up to the natural inclusion map of codomains, 
we have $\chi_{\mathrm{cyc}}^{r}=\kappa^{r}$. 

\subsection{Admissible one-dimensional $p$-adic Lie extensions}
We henceforth assume that $\mathcal{L}/K$ is an \emph{admissible one-dimensional $p$-adic Lie extension}.
In other words, in addition to the existing assumptions that $p$ is an odd prime, $K$ is a number field,
$\mathcal{L}/K$ is a Galois extension of $K$ such that $\mathcal{L}$ contains the cyclotomic $\Z_{p}$-extension $K_{\infty}$ of $K$ and $[\mathcal{L} : K_{\infty}]$ is finite, we now further assume that $\mathcal{L}$ is totally real.
Clearly, this forces $K$ to be totally real, which had not been assumed previously.

\subsection{Power series and $p$-adic Artin $L$-functions}\label{subsec:power-series-p-adic-L-functions}
Fix a character $\chi \in \Irr_{\Q_{p}^{c}}(\mathcal{G})$. 
Each topological generator $\gamma_{K}$ of  $\Gamma_{K}$ permits the definition of a 
power series $G_{\chi,S}(T) \in \Q_{p}^{c} \otimes_{\Q_{p}} Quot(\Z_{p} \llbracket T \rrbracket )$ 
by starting out from the Deligne-Ribet power series for one-dimensional characters of open subgroups 
of $\mathcal{G}$ (see \cite{MR579702}; also see \cite{ MR525346, MR524276}) 
and then extending to the general case by using Brauer induction (see \cite{MR692344}).
One then has an equality
\[
L_{p,S}(1-s,\chi) = \frac{G_{\chi,S}(u^s-1)}{H_{\chi}(u^s-1)},
\]
where $L_{p,S}(s,\chi)$ denotes the `$S$-truncated $p$-adic Artin $L$-function' attached to $\chi$ constructed by Greenberg \cite{MR692344},
and where, for irreducible $\chi$, one has
\[
H_{\chi}(T) = \left\{\begin{array}{ll} \chi(\gamma_{K})(1+T)-1 & \mbox{ if }  H \subseteq \ker \chi\\
1 & \mbox{ otherwise.}  \end{array}\right.
\]
Now \cite[Proposition 11]{MR2114937} implies that
\[
L_{K,S} : \chi \mapsto \frac{G_{\chi,S}(\gamma_{K}-1)}{H_{\chi}(\gamma_{K}-1)}
\]
is independent of the topological generator $\gamma_{K}$ and lies in 
$\Hom^{\ast}(R_{p}( \mathcal{G}), \mathcal{Q}^{c}(\Gamma_{K})^{\times})$.
Diagram \eqref{eqn:Det_triangle} implies that there is a unique element 
$\Phi_{S} = \Phi_{S}(\mathcal{L}/K) \in \zeta(\mathcal{Q}(\mathcal{G}))^{\times}$
such that
\[
j_{\chi}(\Phi_{S}) = L_{K,S}(\chi)
\]
for every $\chi \in \Irr_{\Q_{p}^{c}}(\mathcal{G})$.

\subsection{The $\mu=0$ hypothesis}\label{subsec:mu=0}

Let $S_{\infty}$ be the set of archimedean places of $K$ and let $S_{p}$ be the set of places of $K$ above $p$.
Let $S_{\ram}=S_{\ram}(\mathcal{L}/K)$ be the (finite) set of places of $K$ that ramify in $\mathcal{L}/K$;
note that  $S_{p} \subseteq S_{\ram}$.
Let $S$ be a finite set of places of $K$ containing $S_{\ram} \cup S_{\infty}$.
Let $M_{S}^{\ab}(p)$ be the maximal abelian pro-$p$-extension of 
$\mathcal{L}$ unramified outside $S$ and let $X_{S}=\Gal(M_{S}^{\ab}(p)/\mathcal{L})$. 
As usual $\mathcal{G}$ acts on $X_{S}$ by $g \cdot x = \tilde{g}x\tilde{g}^{-1}$, 
where $g \in \mathcal{G}$, and $\tilde{g}$ is any lift of $g$ to $\Gal(M_{S}^{\ab}(p)/K)$. 
This action extends to a left action of $\Lambda(\mathcal{G})$ on $X_{S}$.
Since $\mathcal{L}$ is totally real, a result of Iwasawa \cite{MR0349627} shows that 
$X_{S}$ is finitely generated and torsion as a $\Lambda(\Gamma_{0})$-module.

\begin{definition}\label{def:mu=0-hypothesis}
We say that $\mathcal{L}/K$ satisfies the $\mu=0$ hypothesis if $X_{S}$ is finitely generated as a $\Z_{p}$-module.
\end{definition}

The  $\mu=0$ hypothesis is conjecturally always true and is known to hold when $\mathcal L / \Q$ is abelian
as follows from work of Ferrero and Washington \cite{MR528968}.
For the relation to the classical Iwasawa $\mu = 0$ conjecture see \cite[Remark 4.3]{MR3749195}, for instance.
In the sequel, we shall \emph{not} assume the $\mu=0$ hypothesis for $\mathcal{L}/K$ except where explicitly stated.

\subsection{A canonical complex}\label{subsec:canonical-complex}
Let $\mathcal{O}_{\mathcal{L},S}$ denote the ring of integers $\mathcal{O}_{\mathcal{L}}$ in $\mathcal{L}$ localised at all primes above those in $S$.
There is a canonical complex 
\[
C_{S}^{\bullet}(\mathcal{L}/K) := R\Hom(R\Gamma_{\et}(\Spec(\mathcal{O}_{\mathcal{L},S}), \Q_{p} / \Z_{p}), \Q_{p} / \Z_{p}),
\]
where $\Q_{p} / \Z_{p}$ denotes the constant sheaf of the abelian group $\Q_{p} / \Z_{p}$ on the \'{e}tale site
of $\Spec(\mathcal{O}_{\mathcal{L},S})$.
The cohomology groups are
\[
H^{i}(C_{S}^{\bullet}(\mathcal{L}/K)) \simeq \left\{
\begin{array}{lll}
X_{S} & \mbox{ if } & i=-1\\
\Z_{p} & \mbox{ if } & i=0\\
0 & \mbox{ if } & i \neq -1,0.\\
\end{array}
\right.
\]
It follows from \cite[Proposition 1.6.5]{MR2276851} that $C_{S}^{\bullet}(\mathcal{L}/K)$ belongs to $\mathcal{D}^{\perf}\tor(\Lambda(\mathcal{G}))$.
In particular, $C_{S}^{\bullet}(\mathcal{L}/K)$ defines a class $[C_{S}^{\bullet}(\mathcal{L}/K)]$ in $K_{0}(\Lambda(\mathcal{G}), \mathcal{Q}(\mathcal{G}))$.
Note that $C_{S}^{\bullet}(\mathcal{L}/K)$ and the complex used by Ritter and Weiss (as constructed in \cite{MR2114937}) become isomorphic in $\mathcal{D}(\Lambda(\mathcal{G}))$ by
\cite[Theorem 2.4]{MR3072281} (see also \cite{MR3068897} for more on this topic).
Hence it makes no essential difference which of these complexes we use.

\subsection{A reformulation of the equivariant Iwasawa main conjecture (EIMC)}\label{subsec:EIMC-reformulation}
We can now state a slight reformulation of the EIMC given in \cite{MR3749195}.
The relation of this version to the framework \cite{MR2217048} (as used in \cite{MR3091976}) 
will be discussed in \S \ref{subsec:relation-to-framework-five-authors}.
Recall that $p$ is an odd prime and $\mathcal{L} / K$ is an admissible one-dimensional $p$-adic Lie extension.

\begin{conj}[EIMC]\label{conj:EIMC}
There exists $\zeta_{S} \in K_{1}(\mathcal{Q}(\mathcal{G}))$ such that $\partial(\zeta_{S}) = -[C_{S}^{\bullet}(\mathcal{L}/K)]$
and $\nr(\zeta_{S}) = \Phi_{S}$.
\end{conj}

It is also conjectured that $\zeta_{S}$ is unique, but we shall not be concerned with this issue here.
Moreover, it can be shown that the truth of Conjecture \ref{conj:EIMC} is independent of the choice of $S$, provided that $S$ is finite and contains $S_{\ram} \cup S_{\infty}$.
Crucially, this version of the EIMC does not require the $\mu=0$ hypothesis for its formulation.
The following theorem has been shown independently by Ritter and Weiss \cite{MR2813337} and Kakde \cite{MR3091976}.

\begin{theorem}\label{thm:EIMC-with-mu}
If $\mathcal{L}/K$ satisfies the $\mu=0$ hypothesis then the EIMC holds for $\mathcal{L}/K$.
\end{theorem}

By considering the cases in which the $\mu=0$ hypothesis is known, we obtain the following corollary
(see \cite[Corollary 4.6]{MR3749195} for further details).

\begin{corollary}
Let $\mathcal{P}$ be a Sylow $p$-subgroup of $\mathcal{G}$.
If $\mathcal{L}^{\mathcal{P}}/\Q$ is abelian then the EIMC holds for $\mathcal{L}/K$.
\end{corollary}

In \cite{MR3749195}, the present authors prove the EIMC in a number of cases where the $\mu=0$ hypothesis is not known.
In \S \ref{sec:unconditional}, these will be combined with main results of the present article (see \S \ref{subsec:statement-of-main-results})
to give unconditional proofs of the non-abelian Brumer--Stark conjecture in many new cases. For now, we only note the following result which relies heavily on a result of Ritter and Weiss \cite[Theorem 16]{MR2114937}. 

\begin{theorem}[{\cite[Theorem 4.12]{MR3749195}}]\label{thm:EIMC-p-does-not-divide-order-of-H}
If $p \nmid |H|$ then the EIMC holds for $\mathcal{L}/K$. 
\end{theorem}

\section{Statement of the main theorem and corollary}\label{sec:statement-of-main-theorem-and-corollary}

\subsection{$p$-adic Artin $L$-functions and the interpolation property}\label{subsec:interpolation-property}
Let $K$ be a totally real number field and let $G_{K}=\Gal(K^{\mathrm{c}}/K)$ be its absolute Galois group.
Let $p$ be an odd prime and let $S_{p}$ denote the set of places of $K$ above $p$.
Let $S_{\infty}$ denote the set of archimedean places of $K$ and let 
$S$ be a finite set of places of $K$ such that $S_{p} \cup S_{\infty} \subseteq S$.

Let $\chi \in \Irr_{\Q_{p}^{c}}(G_{K})$ and let $K_{\chi}$ be the extension of $K$ attached to $\chi$;
thus $\chi$ may be considered as a character attached to a faithful representation of $\Gal(K_{\chi}/K)$
and $K_{\chi}/K$ is of finite degree since $\chi$ has open kernel.
Assume that $K_{\chi}$ is totally real.
Let $\iota:\C_{p} \rightarrow \C$ be a choice of field isomorphism and 
let $L_{S}(s,\iota \circ \chi)$ denote the $S$-truncated Artin $L$-function attached to $\iota \circ \chi \in \Irr_{\C}(\Gal(K_{\chi}/K))$.
For $r \in \Z$ with $r \geq 1$ let $L_{S}(1-r, \chi) = \iota^{-1}(L_{S}(1-r,\iota \circ \chi))$, which is in fact independent of the choice of $\iota$
(compare with the discussion of \S \ref{subsec:L-values}).
If $\chi$ is one-dimensional then for  $r \geq 1$ we have
\begin{equation} \label{eqn:interpolation-property}
L_{p,S}(1-r, \chi) = L_{S}(1-r,\chi\omega^{-r}),
\end{equation}
where $\omega : \Gal(K(\zeta_{p})/K)) \longrightarrow \mu_{p-1} \subseteq \Z_{p}^{\times}$ is the Teichm\"{u}ller character.
Using Brauer induction, \eqref{eqn:interpolation-property} can be extended to the case where $\chi$ is of arbitrary degree 
provided that $r \geq 2$ (see \cite[\S 4]{MR692344}).
However, if $\chi(1)>1$ and $r=1$,
this argument fails due to the potential presence of trivial zeros.
Nevertheless, it seems plausible that the identity
\begin{equation}\label{eqn:values-padic-complex}
        L_{p,S}(0, \chi) = L_S(0, \chi \omega^{-1})
\end{equation}
holds in general.
As both sides are well-behaved with respect to direct sum, inflation and induction of characters, one can show that
\eqref{eqn:values-padic-complex} does hold when $\chi$ is a monomial character, i.e., a character induced from a one-dimensional character of a subgroup
(also see the discussion in \cite[\S 2]{MR656068}).
From recent work of Burns \cite[Theorem 5.2 (i)]{burns-p-adic} it follows that the left hand side of \eqref{eqn:values-padic-complex}
vanishes whenever the right hand side does.

\subsection{Equivariant $p$-adic Artin $L$-values}
Let $L/K$ be a finite Galois CM-extension of number fields with Galois group $G$.
Let $j \in G$ denote complex conjugation. 
Let $L^{+}=L^{\langle j \rangle}$ be the maximal totally real subfield of $L$ and let $G^{+}=\Gal(L^{+}/K) \simeq G/\langle j \rangle$.
Let $p$ be an odd prime.
Recall that $\chi \in \Irr_{\C}(G)$ or $\Irr_{\C_{p}}(G)$ is said to be \emph{even} when $\chi(j) = \chi(1)$, 
and \emph{odd} when $\chi(j) = -\chi(1)$.
Let $S$ be a finite set of places of $K$ such that $S_{p} \cup S_{\infty} \subseteq S$ and
let $T$ be a second set of finite places of $K$ such that $S \cap T = \emptyset$.
In this situation, we define $p$-adic Stickelberger elements by
\begin{eqnarray*}
\theta_{p,S}^{T}(L/K)  =  \theta_{p,S}^{T} & := & (\theta_{p,S,\chi}^{T})_{\chi \in \Irr_{\C_{p}}(G)},\\
\theta_{p,S,\chi}^{T} & := & \left\{ \begin{array}{ll} 0 & \mbox{ if } \chi \mbox{ is even}\\
                                        \delta_{T}(0,\chi) \cdot L_{p,S}(0,\check{\chi} \omega) & \mbox{ if } \chi \mbox{ is odd}.
                                        \end{array} \right.
\end{eqnarray*}

A finite group is said to be monomial if each of its (complex) irreducible characters is monomial.
Note that every finite metabelian or supersoluble group is monomial by \cite[\S 4.4, Theorem 4.8 (1)]{MR1984740}).

\begin{lemma} \label{lem:p-adic-vs-complex-Stickelberger}
If $G^{+}$ is monomial then $\theta_{p,S}^{T} = \theta_{S}^{T}$.
\end{lemma}

\begin{proof}
Dropping $\iota$ from the notation (as we may), 
we recall from \S \ref{subsec:L-values} that 
\[
\theta_{S}^{T} = (\delta_{T}(0,\chi) \cdot L_{S}(0,\check{\chi}))_{\chi \in \Irr_{\C_{p}}(G)}.
\]
Thus for $\chi \in \Irr_{\C_{p}}(G)$ it suffices to show that $L_{S}(0,\check{\chi})=0$ when $\chi$ is even and
$L_{S}(0,\check{\chi})=L_{p,S}(0,\check{\chi} \omega)$ when $\chi$ is odd.
Since $\chi$ is even if and only if $\check{\chi}$ is even, we may replace  $\check{\chi}$ by $\chi$ in the previous sentence.

If $\chi$ is the trivial character then \cite[Chapter I, Proposition 3.4]{MR782485} shows that the order of vanishing of
$L_{S}(s,\chi)$ at $s=0$ is $|S|-1 \geq 1$, and so $L_{S}(0,\chi)=0$.
If $\chi$ is even and non-trivial the argument given in \cite[top of p.\ 71]{MR782485} again shows that $L_{S}(0,\chi)=0$.

Suppose that $\chi$ is odd and set $\psi = \chi \omega$.
Then it suffices to show that $L_{p,S}(0,\psi)=L_{S}(0,\psi \omega^{-1})$.
Moreover, $\psi$ is even and so may be considered as a character of the monomial group $G^{+}$.
Hence the desired equality now follows from the discussion in \S \ref{subsec:interpolation-property} and the appropriate
substitution of symbols.
\end{proof}

\subsection{Statement of the main theorem and corollary}\label{subsec:statement-of-main-results}
We are now in a position to state the main results of this article.
In \S \ref{sec:unconditional}, these will be combined with the authors' previous work on the EIMC \cite{MR3749195}
to give unconditional proofs of the non-abelian Brumer--Stark conjecture in many new cases.

\begin{theorem}\label{thm:EIMC-implies-BS}
Let $L/K$ be a finite Galois CM-extension of number fields with Galois group $G$.
Let $S$ and $T$ be two finite sets of places of $K$ satisfying $\Hyp(S,T)$.
Let $p$ be an odd prime and
let $L(\zeta_p)^{+}_{\infty}$ be the cyclotomic $\Z_{p}$-extension of $L(\zeta_p)^{+}$.
Suppose that $S_{p} \subseteq S$ and that the EIMC holds for $L(\zeta_p)^{+}_{\infty} / K$. Then
\begin{equation}\label{eqn:strongBS}
    (\theta_{p,S}^{T})^{\sharp} \in \Fitt^{\max}_{\Z_{p}[G]_{-}}((A_{L}^{T})^{\vee}).
\end{equation}
\end{theorem}

\begin{remark}
When the classical Iwasawa $\mu$-invariant attached to $L(\zeta_{p})_{\infty}$ vanishes, 
Theorem \ref{thm:EIMC-implies-BS} recovers \cite[Theorem 4.5]{MR3072281},
which in turn is the non-abelian analogue of \cite[Theorem 6.5]{MR3383600}. 
The main reason why some version of the $\mu=0$ hypothesis is assumed in these results
is of course to ensure that the EIMC holds (see also \cite[Remark 4.3]{MR3749195}).
However, both results use a version of the EIMC that requires a $\mu=0$ hypothesis even for its formulation.
For this reason our proof is very different from (though partly inspired by) those in
\cite{MR3383600,MR3072281}. 
Note that in \cite[\S 4]{MR3072281} the identity \eqref{eqn:values-padic-complex} is implicitly assumed to hold.
\end{remark}

\begin{corollary}\label{cor:BS-monomial}
Let $L/K$ be a finite Galois CM-extension of number fields.
Let $p$ be an odd prime and 
let $S$ be a finite set of places of $K$ such that $S_{p} \cup S_{\ram}(L/K) \cup S_{\infty} \subseteq S$.
If $\Gal(L^{+}/K)$ is monomial and the EIMC holds for $L(\zeta_p)^{+}_{\infty} / K$ then both $BS(L/K,S,p)$ and $B(L/K,S,p)$ are true.
\end{corollary}

\begin{proof}
This is the combination of Proposition \ref{prop:SBS_implies_BS}, 
Lemma \ref{lem:p-adic-vs-complex-Stickelberger}, and  
Theorem \ref{thm:EIMC-implies-BS}.  
\end{proof}

\section{Iwasawa algebras and the main conjecture revisited}\label{subsec:iwasawa-algs-and-mcs-revisited}

\subsection{Certain maps between Iwasawa algebras} \label{subsec:Certain-maps}
We assume the setup and notation of \S \ref{subsec:certain-one-dim-Lie}--\S \ref{subsec:cyclotomic-char}.
Fix a character $\chi \in \Irr_{\Q_{p}^{c}}(\mathcal{G})$ and let $\eta$ be an irreducible constituent of $\res^{\mathcal{G}}_{H} \chi$.
Let $F/\Q_{p}$ be a finite extension over which both characters $\chi$ and $\eta$ have realisations.
Recalling the notation of \S \ref{subsec:idempotents} and \S \ref{subsec:cyclotomic-char} in particular, 
for $r \in \Z$ we define maps
\[
j_{\chi}^{r}: \zeta(\mathcal{Q}^{F} (\mathcal{G})) \twoheadrightarrow \zeta(\mathcal{Q}^{F} (\mathcal{G})e_{\chi}) \simeq \mathcal{Q}^{F}(\Gamma_{\chi}) \rightarrow  \mathcal{Q}^{F}(\Gamma_{K}),
\]
where the last arrow is induced by mapping $\gamma_{\chi}$ to 
$(u^{r}\gamma_{K})^{w_{\chi}}$.
Note that $j_{\chi}^{0} = j_{\chi}$ (see \S \ref{subsec:dets-and-nr}).

Now assume that $\zeta_{p} \in \mathcal{L}$.
For $s \in \Z$ let $x \mapsto t_{\mathrm{cyc}}^{s}(x)$ and $x \mapsto t_{\omega}^{s}(x)$
be the automorphisms on $\mathcal{Q}^{F}(\mathcal{G})$ induced by 
$g \mapsto \chi_{\mathrm{cyc}}^{s}(g)g$ 
and $g \mapsto \omega^{s}(g)g$ for $g \in \mathcal{G}$, respectively. 

\begin{lemma}  \label{lem:jchi-tcyc-composition}
Let $r,s \in \Z$. Then for every $x \in \zeta(\mathcal{Q}^{F}(\mathcal{G}))$
we have $j_{\chi}^{r}( t_{\mathrm{cyc}}^{s}(x)) = j_{\chi \omega^{s}}^{r+s}(x)$. 
\end{lemma}

\begin{proof}
It follows easily from the definitions that $w_{\chi} = w_{\chi \omega}$.
We claim that
\begin{equation} \label{eqn:gammachi}
t_{\omega}^{1}(\gamma_{\chi \omega}) = \gamma_{\chi}.
\end{equation}
Write $\gamma_{\chi \omega} = g_{\chi \omega} c_{\chi \omega}$ with
$g_{\chi \omega} \in \mathcal{G}$ and $c_{\chi \omega} \in (F[H] e_{\chi \omega})^{\times}$
where $g_{\chi \omega}$ and $c_{\chi \omega}$ satisfy the defining properties of $\gamma_{\chi\omega}$ as given in \S \ref{subsec:idempotents}. 
Put $g_{\chi}' := g_{\chi \omega}$ and
$c_{\chi}' := \omega(g_{\chi \omega}) t_{\omega}^{1}(c_{\chi \omega})$.
It is then easily checked that $g_{\chi}' c_{\chi}'$ has the defining properties of $\gamma_{\chi}$, and thus
$t_{\omega}^{1}(\gamma_{\chi \omega}) = g_{\chi}' c_{\chi}' = \gamma_{\chi}$. 
This establishes \eqref{eqn:gammachi}.
Recalling that $u = \kappa(\gamma_{K})$ we compute
\[
t_{\mathrm{cyc}}^{1}(\gamma_{\chi \omega}) 
= \chi_{\mathrm{cyc}}(g_{\chi \omega}) g_{\chi \omega} t_{\omega}^{1} (c_{\chi \omega})
= u^{w_{\chi \omega}} t_{\omega}^{1}(\gamma_{\chi \omega})
=  u^{w_{\chi}} \gamma_{\chi},
\]
where we have used \eqref{eqn:gammachi} for the last equality. Finally, we have that
\[
j_{\chi}^{r} (t_{\mathrm{cyc}}^{s}(\gamma_{\chi \omega^{s}}))
= j_{\chi}^{r} (u^{w_{\chi} \cdot s} \gamma_{\chi})
=  u^{w_{\chi} \cdot s}(u^{r} \gamma_{K})^{w_{\chi}}
= (u^{r+s} \gamma_{K})^{w_{\chi}}
= j_{\chi \omega^{s}}^{r+s}(\gamma_{\chi \omega^{s}})
\]
for every $r,s \in \Z$ as desired.
\end{proof}

\subsection{Relation to the framework of \cite{MR2217048}}\label{subsec:relation-to-framework-five-authors}
We now discuss Conjecture \ref{conj:EIMC} within the framework of the theory of \cite[\S 3]{MR2217048}.
Let $p$ be an odd prime and let  $\mathcal L / K$ be an admissible one-dimensional $p$-adic Lie extension.
Let
\[
\pi: \mathcal{G} \rightarrow \GL_{n}(\mathcal{O})
\]
be a continuous homomorphism, where $\mathcal{O}=\mathcal{O}_{F}$ denotes the ring of integers of a finite extension $F$ of $\mathbb{Q}_{p}$
and $n$ is some positive integer.
There is a ring homomorphism
\begin{equation} \label{eqn:first_Phi}
\Phi_{\pi}: \Lambda(\mathcal{G}) \rightarrow M_{n\times n}(\Lambda^{\mathcal{O}}(\Gamma_{K}))
\end{equation}
induced by the continuous group homomorphism
\begin{eqnarray*}
\mathcal{G} & \rightarrow & (M_{n \times n}(\mathcal{O}) \otimes_{\Z_p} \Lambda(\Gamma_{K}))\mal = \GL_{n}(\Lambda^{\mathcal{O}}(\Gamma_{K}))\\
\sigma & \mapsto & \pi(\sigma) \otimes \overline{\sigma},
\end{eqnarray*}
where $\overline{\sigma}$ denotes the image of $\sigma$ in $\mathcal{G} / H = \Gamma_{K}$. 
By \cite[Lemma 3.3]{MR2217048} the
homomorphism \eqref{eqn:first_Phi} extends to a ring homomorphism
\[
\Phi_{\pi}: \mathcal{Q}(\mathcal{G}) \rightarrow M_{n\times n}(\mathcal{Q}^{F}(\Gamma_{K}))
\]
and this in turn induces a homomorphism
\[
\Phi_{\pi}': K_{1}(\mathcal{Q}(\mathcal{G})) \rightarrow 
K_{1}(M_{n\times n}(\mathcal{Q}^{F}(\Gamma_{K}))) = \mathcal{Q}^{F}(\Gamma_{K})\mal.
\]
Let $\aug: \Lambda^{\mathcal{O}}(\Gamma_{K}) \twoheadrightarrow \mathcal{O}$ be the augmentation map and put $\mathfrak{p} = \ker(\aug)$.
Writing $\Lambda^{\mathcal{O}}(\Gamma_{K})_{\mathfrak{p}}$ for the localisation of $\Lambda^{\mathcal{O}}(\Gamma_{K})$ at $\mathfrak{p}$, it is clear that $\aug$ naturally extends to a homomorphism $\aug: \Lambda^{\mathcal{O}}(\Gamma_{K})_{\mathfrak{p}} \rightarrow F$.
One defines an evaluation map
\begin{equation} \label{eqn:evaluation-map}
\begin{array}{rcl}
\phi: \mathcal{Q}^{F}(\Gamma_{K}) & \rightarrow & F \cup \{\infty\}\\
x & \mapsto & \left\{ \begin{array}{ll} \aug (x) & \mbox{ if } x \in \Lambda^{\mathcal{O}}(\Gamma_{K})_{\mathfrak{p}}\\
\infty & \mbox{ otherwise}. \end{array} \right.
\end{array}
\end{equation}
It is straightforward to show that for $r \in \Z$ we have
\begin{equation}\label{eq:PhiS-jr-p-adic}
\phi(j_{\chi}^{r}(\Phi_{S})) = L_{p,S}(1-r, \chi). 
\end{equation}

If $\zeta$ is an element of $K_{1}(\mathcal{Q}(\mathcal{G}))$, we define $\zeta(\pi)$ to be $\phi(\Phi_{\pi}'(\zeta))$.
Conjecture \ref{conj:EIMC} now implies that there is an element $\zeta_{S} \in K_{1}(\mathcal{Q}(\mathcal{G}))$ such that
$\partial(\zeta_{S}) = -[C_{S}^{\bullet}(\mathcal{L}/K)]$ and for each $r \geq 1$ divisible by $p-1$
and every irreducible Artin representation $\pi_{\chi}$ of $\mathcal{G}$ with character $\chi$ we have
\[
\zeta_{S}(\pi_{\chi}\kappa^{r}) = \phi(j_{\chi}^{r}(\Phi_{S})) = L_{p,S}(1-r, \chi) = L_{S}(1-r,\chi),
\]
where the first equality follows from \cite[Lemma 2.3]{MR2822866} (for the last equality see \S \ref{subsec:interpolation-property}).

\subsection{Non-commutative Fitting invariants over Iwasawa algebras}\label{subsec:noncomm-fitt-over-Iwasawa}
Let $p$ be an odd prime and let $\mathcal{G} = H \rtimes \Gamma$ be a one-dimensional $p$-adic Lie group.
Let $\Gamma' \simeq \Z_{p}$ be a normal subgroup of $\mathcal{G}$ such that $\Gamma' \cap H = 1$.
Then $\Gamma'$ is open in $\mathcal{G}$ and we set $G := \mathcal{G} / \Gamma'$.
Thus every irreducible character $\chi$ of $G$ may be viewed as an irreducible character of $\mathcal{G}$ with open kernel.
For any such character, let $e(\chi) := \chi(1) |G|^{-1} \sum_{g \in G} \chi(g^{-1}) g$ be the corresponding primitive central idempotent of
$\Q_{p}^{c}[G]$.
Let $\Lambda(\mathcal{G}):=\Z_{p} \llbracket \mathcal{G} \rrbracket$ be the Iwasawa algebra of $\mathcal{G}$. 

For the next result we note
that the maps $j_{\chi}$ and $\phi$
of \S \ref{subsec:Certain-maps} and \S
\ref{subsec:relation-to-framework-five-authors} are purely algebraic in
nature and thus do not depend on the underlying Galois extension.

\begin{prop}[{\cite[Theorem 6.4]{MR2609173}}]\label{prop:Fitting-descent}
Let $M$ be a finitely presented $\Lambda(\mathcal{G})$-module.
Let $\lambda \in \Fitt_{\Lambda(\mathcal{G})}^{\max}(M)$ and set
\[
     \overline{\lambda} \quad := 
     \sum_{\chi \in \Irr_{\Q_{p}^{c}}(G)} \phi(j_{\chi}(\lambda)) e(\chi) \in \zeta(\Q_{p}[G]).
\]
Then $ \overline{\lambda} \in  \Fitt_{\Z_{p}[G]}^{\max}(M_{\Gamma'})$,
\end{prop}

\section{Proof of Theorem \ref{thm:EIMC-implies-BS}: working at the finite level}\label{sec:proof-finite-level}

\subsection{\'{E}tale cohomology}
Let $L/K$ be a finite Galois extension of number fields with Galois group $G$ and recall the notation 
of \S \ref{subsec:ray-class-groups}.
We fix two finite disjoint nonempty sets $S$ and $T$ of places of $K$ such that $S$ contains $S_{\infty}$. 
We put $U_{S} := \Spec(\mathcal{O}_{L,S})$ and $Z_{T} := \Spec(\mathcal{O}_{L,S} / \mathfrak{M}_{L}^{T})$. 
Let $\mathbb{G}_{m,X}$ denote the \'{e}tale sheaf defined by the group of units of a scheme $X$. 
The closed immersion $\iota: Z_{T} \rightarrow U_{S}$ induces a canonical morphism 
$\mathbb{G}_{m,U_{S}} \rightarrow \iota_{\ast} \mathbb{G}_{m,Z_{T}}$, which can be shown to be surjective 
by considering stalks. 
Let $\mathbb{G}_{m,U_{S}}^{T}$ denote the kernel of this morphism; then 
we have an exact sequence of \'{e}tale sheaves
\begin{equation} \label{eqn:ses-etale-sheaves}
0 \longrightarrow \mathbb{G}_{m,U_{S}}^{T} \longrightarrow \mathbb{G}_{m,U_{S}}
\longrightarrow \iota_{\ast} \mathbb{G}_{m,Z_{T}} \longrightarrow 0.
\end{equation}
If $w$ is a finite place of $L$, we let $\mathcal{O}_{L,w}$ be the localisation of $\mathcal{O}_{L}$ at $w$.
We denote the field of fractions of the Henselisation $\mathcal{O}_{L,w}^{h}$ of $\mathcal{O}_{L,w}$ by $L_{w}$.
If $w$ is archimedean, we let $L_{w}$ be the completion of $L$ at $w$. In both cases we let $\Br(L_{w})$ be the Brauer group of $L_{w}$.

The main purpose of this subsection is to generalise the following result.

\begin{prop}\label{prop:cohomology_of_Gm}
Let $S$ be a finite set of places of $K$ containing $S_{\infty}.$ Then
\begin{eqnarray*}
H^{0}_{\et}(U_{S}, \mathbb{G}_{m,U_{S}}) & \simeq & E_{L,S},\\
H^{1}_{\et}(U_{S}, \mathbb{G}_{m,U_{S}}) & \simeq & \cl_{L,S},
\end{eqnarray*}
there is an exact sequence
\[
0 \longrightarrow H^{2}_{\et}(U_{S}, \mathbb{G}_{m,U_{S}}) \longrightarrow \bigoplus_{w \in S(L)} \Br(L_{w}) \longrightarrow \Q / \Z
\longrightarrow H^{3}_{\et}(U_{S}, \mathbb{G}_{m,U_{S}}) \longrightarrow 0,
\]
and for $i \geq 4$ we have
\[
H^{i}_{\et}(U_{S}, \mathbb{G}_{m,U_{S}}) \simeq \bigoplus_{w \in S_{\infty}(L) \atop w \mbox{ \tiny{real}}}
H^{i}_{\et}(\Spec(L_{w}), \mathbb{G}_{m,\Spec(L_{w})}).
\]
\end{prop}

\begin{proof}
This is \cite[Chapter II, Proposition 2.1]{MR2261462}.
\end{proof}

\begin{lemma}\label{lem:qis-i-zt}
Let $S$ and $T$ be as above. 
Then $R\Gamma_{\et}(U_{S}, \iota_{\ast}\mathbb{G}_{m,Z_{T}}) \simeq R\Gamma_{\et}(Z_{T}, \mathbb{G}_{m,Z_{T}})$.
Moreover, $H^{0}_{\et}(Z_{T},\mathbb{G}_{m,Z_{T}}) \simeq (\mathcal{O}_{L,S} / \mathfrak{M}_{L}^{T})^{\times}$ and 
$H^{i}_{\et}(Z_{T},\mathbb{G}_{m,Z_{T}})=0$ for $i \geq 1$.
\end{lemma}

\begin{proof}
For a finite field $\mathbb{F}$, the cohomology of $R\Gamma_{\et}(\Spec(\mathbb{F}), \mathbb{G}_{m,\Spec(\mathbb{F})})$
vanishes outside degree $0$ and $H_{\et}^{0}(\Spec(\mathbb{F}), \mathbb{G}_{m,\Spec(\mathbb{F})}) \simeq \F^{\times}$.
Moreover, we have an isomorphism 
\[
R\Gamma_{\et}(Z_{T}, \mathbb{G}_{m,Z_{T}})
\simeq \bigoplus_{w \in T(L)} R\Gamma_{\et}(\Spec(L(w)), \mathbb{G}_{m,\Spec(L(w))}),
\]
where $L(w)$ denotes the finite field $\mathcal{O}_{L} / \mathfrak{P}_{w}$, and so the second claim follows.
Note that the natural map 
$H^{0}_{\et}(Z_{T},\mathbb{G}_{m,Z_{T}}) \rightarrow H^{0}_{\et}(U_{S},\iota_{*}\mathbb{G}_{m,Z_{T}})$ is in fact an isomorphism.
Furthermore, the functor $\iota_{\ast}$ is exact for the \'{e}tale topology by \cite[Chapter II, Corollary 3.6]{MR559531}.
Thus the universal property of derived functors gives the first claim.
\end{proof}

Let $\Sigma$ be a subset of $S$ containing $S_{\infty}$.
We shall consider the natural map 
\begin{equation}\label{eqn:define-psi}
\psi_{\Sigma,S}^{T} = \psi_{\Sigma,S}^{T}(L):
R\Gamma_{\et}(U_{\Sigma}, \mathbb{G}_{m, U_{\Sigma}}^{T}) \longrightarrow R\Gamma_{\et}(U_{S}, \mathbb{G}_{m, U_{S}}^{T}).
\end{equation}
Here, the set $T$ may be empty, in which case we put $\mathbb{G}_{m, U_{S}}^{\emptyset} := \mathbb{G}_{m, U_{S}}$ 
and similarly with $\mathbb{G}_{m, U_{\Sigma}}$.
Sequence \eqref{eqn:ses-etale-sheaves} and Lemma \ref{lem:qis-i-zt} for $S$ and $\Sigma$ 
induce a commutative diagram
\begin{equation}\label{eqn:change-sigma-to-S}
\xymatrix{
{R\Gamma_{\et}(U_{\Sigma}, \mathbb{G}_{m, U_{\Sigma}}^{T})} \ar[r] \ar[d]^{\psi_{\Sigma,S}^{T}} &
R\Gamma_{\et}(U_{\Sigma}, \mathbb{G}_{m, U_{\Sigma}}) \ar[r] \ar[d]^{\psi_{\Sigma,S}} & R\Gamma_{\et}(Z_{T}, \mathbb{G}_{m,Z_{T}}) \ar@{=}[d]\\
{R\Gamma_{\et}(U_{S}, \mathbb{G}_{m, U_{S}}^{T})} \ar[r] &
R\Gamma_{\et}(U_{S}, \mathbb{G}_{m, U_{S}}) \ar[r] & R\Gamma_{\et}(Z_{T}, \mathbb{G}_{m,Z_{T}})
} 
\end{equation}
where the rows are exact triangles.
Let $D_{\Sigma,S}^{T} = D_{\Sigma,S}^{T}(L)$ be the cone of $\psi_{\Sigma,S}^{T}$. The diagram shows that
$D_{\Sigma,S}^{T}$ does not in fact depend on $T$ and thus we denote it by $D_{\Sigma,S}$.

\begin{prop}\label{prop:cohomology-of-cone}
Let $S,T$ and $\Sigma$ be as above. We have
\[
H^{i}(D_{\Sigma,S}^{T}(L)) \simeq H^{i}(D_{\Sigma,S}(L)) \simeq \left\{
\begin{array}{lll}
\Z[S(L)-\Sigma(L)] & \mbox{ if } & i=0\\
\Q / \Z[S(L)-\Sigma(L)] & \mbox{ if } & i=2\\
0 & \mbox{ if } & i\not=0,2.\\
\end{array}
\right.
\]
\end{prop}

\begin{proof}
As $D_{\Sigma,S}^{T}=D_{\Sigma,S}$ does not depend on $T$, we can and do assume that $T$ is empty.
Let $Z_{\Sigma,S} := U_{\Sigma} - U_{S}$; then $Z_{\Sigma,S}$ is a closed subscheme of $U_{\Sigma}$ and
by \cite[Chapter III, Proposition 1.25]{MR559531} we have an isomorphism
\[
D_{\Sigma,S} \simeq R\Gamma_{Z_{\Sigma,S}}(U_{\Sigma}, \mathbb{G}_{m, U_{\Sigma}})[-1],
\]
where the righthand side denotes cohomology with support on $Z_{\Sigma,S}$.
We now apply \cite[Chapter III, Corollary 1.28]{MR559531} and \cite[Chapter II, Proposition 1.5]{MR2261462} to
obtain the desired result.
\end{proof}

\begin{prop}\label{prop:cohomology_of_GmT}
Let $S$ and $T$ be as above. Then
\[
H^{i}_{\et}(U_{S}, \mathbb{G}_{m,U_{S}}^{T}) \simeq \left\{
\begin{array}{lll}
E_{L,S}^{T} & \mbox{ if } & i=0\\
\cl_{L,S}^{T} & \mbox{ if } & i=1\\
H^{i}_{\et}(U_{S}, \mathbb{G}_{m,U_{S}}) & \mbox{ if } & i \geq 2.\\
\end{array}
\right.
\]
\end{prop}

\begin{proof}
By Proposition \ref{prop:cohomology_of_Gm} and Lemma \ref{lem:qis-i-zt}, the 
long exact sequence of cohomology groups induced by \eqref{eqn:ses-etale-sheaves} yields an exact sequence
\[
0 \longrightarrow H^{0}_{\et}(U_{S}, \mathbb{G}_{m,U_{S}}^{T}) \longrightarrow E_{L,S} \longrightarrow (\mathcal{O}_{L,S} / \mathfrak{M}_{L}^{T})\mal
\longrightarrow H^{1}_{\et}(U_{S}, \mathbb{G}_{m,U_{S}}^{T}) \longrightarrow \cl_{L,S} \longrightarrow 0
\]
and isomorphisms $ H^{i}_{\et}(U_{S}, \mathbb{G}_{m,U_{S}}^{T}) \simeq H^{i}_{\et}(U_{S}, \mathbb{G}_{m,U_{S}})$ for all $i \geq 2$.
It follows from this and \eqref{eqn:ray_class_sequence} that 
$H^{0}_{\et}(U_{S}, \mathbb{G}_{m,U_{S}}^{T})  \simeq E_{L,S}^{T}$, and that
$H^{1}_{\et}(U_{S}, \mathbb{G}_{m,U_{S}}^{T})$ and $\cl_{L,S}^{T}$ have the same cardinality;
it remains to show that they are in fact isomorphic. 

We now change notation as follows: let $\Sigma=S$ for the choice of $S$ as in the statement of the proposition; and
enlarge $S$ in such a way that $S$ is finite and disjoint from $T$ and that $\cl_{L,S}^{T}$ vanishes.
The same reasoning as above shows that $H^{1}_{\et}(U_{S}, \mathbb{G}_{m,U_{S}}^{T})$ also vanishes.
Therefore the long exact cohomology sequence induced 
by \eqref{eqn:define-psi} yields an exact sequence
\begin{equation}\label{eqn:seq-induced-by-sigma-S-cone}
0 \longrightarrow E_{L, \Sigma}^T \longrightarrow E_{L,S}^T \longrightarrow
\Z [S(L) - \Sigma(L)] \longrightarrow H_{\et}^{1}(U_{\Sigma},\mathbb{G}_{m,U_{\Sigma}}^{T}) \longrightarrow 0,
\end{equation}
where $H^{0}(D_{\Sigma,S}^{T}(L)) \simeq \Z [S(L) - \Sigma(L)]$ by Proposition \ref{prop:cohomology-of-cone}.
Comparing \eqref{eqn:seq-induced-by-sigma-S-cone} to \eqref{eqn:ray_class_sequence_ZS} yields 
$H_{\et}^{1}(U_{\Sigma},\mathbb{G}_{m,U_{\Sigma}}^{T}) \simeq \cl_{L,\Sigma}^{T}$, as desired.
\end{proof}

\begin{remark}
The proof of Proposition \ref{prop:cohomology_of_GmT} shows that the long exact sequence in cohomology induced by \eqref{eqn:define-psi} 
yields an exact sequence whose first terms coincide with \eqref{eqn:ray_class_sequence_ZS}. 
Similarly, taking global sections in \eqref{eqn:ses-etale-sheaves} gives a long exact sequence
in cohomology whose first terms coincide with \eqref{eqn:ray_class_sequence}.
\end{remark}

\subsection{Flat cohomology}
Let $X$ be an affine scheme and consider the  multiplicative group scheme $\mathbb{G}_{m/X}$ over  $X$.
We let $\mu_{n / X}$ be the kernel of multiplication by $n \in \N$ and for a prime $p$ we put
$\mu_{p^{\infty} /X} := \varinjlim_{j} \mu_{p^{j} / X}$ which is an ind-$X$-group scheme.

By the flat site on $X$ we shall mean the site of all quasi-finite flat schemes of finite presentation over $X$
with the $fppf$-topology; so coverings are surjective families of flat morphisms that are locally of finite presentation.
Note that our definition agrees with the $fpqf$-site in \cite{MR689645}, but differs from the flat site in
\cite{MR559531}; however, this will play no decisive role in the following by \cite[Chapter III, Proposition 3.1]{MR559531}.
The group schemes $\mathbb{G}_{m/X}$ and $\mu_{n / X}$ represent abelian sheaves for the flat site on $X$
and we denote the corresponding cohomology groups by $H^{i}_{\fl}(X,\mathbb{G}_{m})$ and 
$H^{i}_{\fl}(X, \mu_{n})$, respectively.  As $\mathbb{G}_{m}$ is smooth over $X$, we have an isomorphism
\begin{equation} \label{eqn:flat-etale-Gm}
	R\Gamma_{\fl}(X, \mathbb{G}_{m}) \simeq R\Gamma_{\et}(X, \mathbb{G}_{m,X})
\end{equation}
by \cite[Chapter III, Theorem 3.9]{MR559531}. 

For a prime $p$ we write $\mu_{p}(L)$ for the group of $p$-power roots of unity in $L$.
We now specialise to the case $X = U_{L} := U_{S_{\infty}}$ and $p$ odd.
Then $\mu_{p^{\infty} / U_{L}}$ also
represents a sheaf for the flat site on $U_{L}$. As the Kummer sequences
\[
 	0 \longrightarrow \mu_{p^n} \longrightarrow \mathbb{G}_{m} \stackrel{p^n}{\longrightarrow} \mathbb{G}_{m} \longrightarrow 0
\]
are exact for the flat site on $U_L$ for all $n \in \N$, we have an exact triangle
\begin{equation} \label{eqn:flat_triangle}
	R\Gamma_{\fl}(U_{L},\mu_{p^{\infty}}) \longrightarrow R\Gamma_{\fl}(U_{L}, \mathbb{G}_{m})(p) \longrightarrow 
	\Q_{p} \otimes^{\mathbb{L}}_{\Z_{p}} R\Gamma_{\fl}(U_{L}, \mathbb{G}_{m})(p),
\end{equation}
where we use the notation $C^{\bullet}(p) := \Z_{p}\otimes^{\mathbb{L}}_{\Z} C^{\bullet} \in \mathcal{D}(\Z_{p}[G])$
for any complex $C^{\bullet}$ in $\mathcal{D}(\Z[G])$.
By Proposition \ref{prop:cohomology_of_Gm}
and \eqref{eqn:flat-etale-Gm} we thus obtain isomorphisms
\[
	H^{i}_{\fl}(U_L, \mu_{p^{\infty}}) \simeq \left\{ 
		\begin{array}{lll}
		\mu_{p}(L) & \mbox{ if } & i=0\\	
		\Q_{p} / \Z_{p} & \mbox{ if } & i=3\\
		0 & \mbox{ if } & i \not= 0,1,3
		\end{array}\right.
\]
and an exact sequence
\[
	0 \longrightarrow \mathcal{O}_{L}^{\times} \otimes_{\Z} \Q_{p} / \Z_{p} \longrightarrow H^{1}_{\fl}(U_L, \mu_{p^{\infty}})
	\longrightarrow \cl_{L} \otimes_{\Z} \Z_{p} \longrightarrow 0.
\]

\subsection{A reduction step} We now begin the proof of the main theorem of this article.

\begin{proof}[Proof of Theorem \ref{thm:EIMC-implies-BS}]
This proof will occupy the rest of \S \ref{sec:proof-finite-level} and all of \S \ref{sec:proof-infinite-level}.
We first prove a reduction step that will allow us to make certain simplifying assumptions.

Put $L' := L(\zeta_{p})$ and $C := \Gal(L'/L)$. 
Then $C$ is a cyclic group whose order divides $p-1$.
Let $N:A_{L'}^{T} \rightarrow A_{L}^{T}$ and $i:A_{L}^{T} \rightarrow A_{L'}^{T}$ 
denote the homomorphisms induced by the norm and inclusion maps on ideals, respectively.
Then $N \circ i : A_{L}^{T} \rightarrow A_{L}^{T}$ is multiplication by $|C|$ and thus is an isomorphism of $p$-groups.
In particular, $i$ is injective and $N$ is surjective. 
Let $\Delta(C)$ denote the kernel of the augmentation map $\Z[C] \twoheadrightarrow \Z$ which maps each $c \in C$ to $1$. 
The composite map $i \circ N: A^{T}_{L'} \rightarrow A^{T}_{L'}$
(also referred to as a norm map) is given by multiplication by $\sum_{c \in C} c$.
Since the orders of $C$ and $A^{T}_{L'}$ are coprime, $A^{T}_{L'}$ is cohomologically trivial as a $C$-module and thus
$\ker(i \circ N) =  \Delta(C) A_{L'}^{T}$. 
As $i$ is injective we thus have  $\ker(N) =  \Delta(C) A_{L'}^{T}$ and since $N$ is surjective we conclude that it induces an isomorphism
 $(A_{L'}^{T})_{C} \simeq A_{L}^{T}$.
Using standard properties of Pontryagin duals and (co)-invariants, we conclude that 
\[
((A_{L'}^{T})^{\vee})_{C} \simeq ((A_{L'}^{T})^{C})^{\vee} \simeq ((A_{L'}^{T})_{C})^{\vee} \simeq (A_{L}^{T})^{\vee}.
\]
The idempotent $e_{C} := |C|^{-1} \sum_{c \in C} c$ belongs to the group ring $\Z_{p}[\Gal(L'/K)]$ and so
\[
\Fitt^{\max}_{\Z_{p}[G]_{-}}((A_{L}^{T})^{\vee}) = e_{C} \Fitt^{\max}_{\Z_{p}[\Gal(L'/K)]_{-}}((A_{L'}^{T})^{\vee})
\]
by Lemma \ref{lem:Fitting-properties} (iii).
As Stickelberger elements also behave well under base change, i.e., $e_{C} \theta_{p,S}^{T}(L'/K) = \theta_{p,S}^{T}(L/K)$, we may assume without loss of generality that $\zeta_{p} \in L$.
Note that as we are considering the $p$-parts we only need that $E^{T}_{L,S}(p)$ is torsionfree,
as opposed to the stronger requirement that $E^{T}_{L,S}$ is torsionfree.
Since $S_{p} \subseteq S$ and $\Hyp(S,T)$ holds, this hypothesis is unaffected by replacing $L$ with $L(\zeta_{p})$.

For clarity, we now list the assumptions that we shall use for the rest of this proof 
(including the lemmas and propositions proved along the way).
We can make these assumptions either for the reasons just explained or because they are direct consequences 
of our hypotheses. Note that $S_{\ram}(L/K) \cup S_{p} = S_{\ram}(L_{\infty}/K)$.

\begin{assumptions*}
We henceforth assume that $S,T$ are finite sets of places of $K$ and that
(i) $\zeta_{p} \in L$, (ii) $S \cap T = \emptyset \neq T$, (iii)  $S_{p} \cup S_{\ram}(L/K) \cup S_{\infty} \subseteq S$,
and (iv) $E^{T}_{L,S}(p)$ is torsionfree.
\end{assumptions*}

\subsection{Complexes at the finite level}

Now taking $p$-minus parts of sequence \eqref{eqn:ray_class_sequence} for $S = S_{\infty}$ 
yields an exact sequence of $\Z_{p}[G]_{-}$-modules
\begin{equation} \label{eqn:minus-rayclass-sequence}
0 \longrightarrow \mu_{p}(L) \longrightarrow (\mathcal{O}_{L} / \mathfrak{M}_{L}^{T})^{\times} (p)^{-} \longrightarrow A_{L}^{T}
\longrightarrow A_{L} \longrightarrow 0.
\end{equation}
The middle arrow $(\mathcal{O}_{L} / \mathfrak{M}_{L}^{T})\mal (p)^{-} \rightarrow A_{L}^{T}$ defines a complex
$C^{T \bullet}(L/K)$ in $\mathcal{D}(\Z_{p}[G]_{-})$, where we place the first module in degree $0$.
For complexes $C^{\bullet}$ in $\mathcal{D}(\Z_{p}[G])$ 
we put $C^{\bullet -} := \Z_{p}[G]_{-} \otimes^{\mathbb{L}}_{\Z_{p}[G]} C^{\bullet} \in \mathcal{D}(\Z_{p}[G]_{-})$.
Note that taking $p$-minus parts is an exact functor as $p$ is odd.

\begin{prop}\label{prop:coh-interpretation-L/K}
There are isomorphisms
\[
C^{T \bullet}(L/K) \simeq R\Gamma_{\et}(U_{L}, \mathbb{G}_{m,U_{L}})(p)^{-}
\simeq R\Gamma_{\fl}(U_{L}, \mu_{p^{\infty}})^{-}
\]
in $\mathcal{D}(\Z_{p}[G]_{-})$. In particular, the isomorphism class of $C^{T \bullet}(L/K)$ does not depend on $T$.
\end{prop}

\begin{proof}
Proposition \ref{prop:cohomology_of_Gm} describes the cohomology of $R\Gamma_{\et}(U_{L}, \mathbb{G}_{m,U_{L}})(p)^{-}$ as follows.
First, we have isomorphisms
\[
H^{0}_{\et}(U_{L}, \mathbb{G}_{m,U_{L}})(p)^{-} \simeq \mu_{p}(L), \quad
H^{1}_{\et}(U_{L}, \mathbb{G}_{m,U_{L}})(p)^{-} \simeq A_{L}.
\]
As $L$ is totally complex,
we have $H^{i}_{\et}(U_{L}, \mathbb{G}_{m,U_{L}}) = 0$ for every $i \geq 4$. Finally, we have an exact sequence
\[
0 \longrightarrow H^{2}_{\et}(U_{L}, \mathbb{G}_{m,U_{L}}) \longrightarrow \bigoplus_{w \in S_{\infty}(L)} \Br(L_{w}) \longrightarrow \Q / \Z
\longrightarrow H^{3}_{\et}(U_{L}, \mathbb{G}_{m,U_{L}}) \longrightarrow 0.
\]
However, the Brauer groups $\Br(L_{w}) = \Br(\C)$ vanish for all archimedean places $w$ of $L$ and $(\Q / \Z)^{-} = 0$.
We therefore have $H^{i}_{\et}(U_{L}, \mathbb{G}_{m,U_{L}})(p)^{-} = 0$ for every $i \geq 2$. 
Using Proposition \ref{prop:cohomology_of_GmT} and the assumption that $E_{L,S}^{T}(p)$ is torsionfree,
we likewise find that
\[
H^{i}_{\et}(U_{L}, \mathbb{G}_{m,U_{L}}^{T})(p)^{-} \simeq \left\{ \begin{array}{lll}
A_{L}^{T} & \mbox{ if } & i = 1\\
0 & \mbox{ if } & i \not=1.
\end{array} \right.
\]
Now sequence \eqref{eqn:ses-etale-sheaves} and Lemma \ref{lem:qis-i-zt} together give an exact triangle
\[ \xymatrix{
{R\Gamma_{\et}(U_{L}, \mathbb{G}_{m,U_{L}}^{T})(p)^{-}} \ar[r] \ar[d]^{\simeq} &
R\Gamma_{\et}(U_{L}, \mathbb{G}_{m,U_{L}})(p)^{-} \ar[r] & R\Gamma_{\et}(U_{L}, \iota_{\ast} \mathbb{G}_{m,Z_{T}})(p)^{-} \ar[d]^{\simeq}\\
{A_{L}^{T} [-1]} & & (\mathcal{O}_{L} / \mathfrak{M}_{L}^{T})^{\times}(p)^{-}
}\]
and thus we obtain the first required isomorphism in $\mathcal{D}(\Z_{p}[G]_{-})$.

Equation \eqref{eqn:flat-etale-Gm} with $X=U_{L}$ shows that
$R\Gamma_{\et}(U_{L}, \mathbb{G}_{m,U_{L}}) \simeq  R\Gamma_{\fl}(U_{L}, \mathbb{G}_{m})$ in $\mathcal{D}(\Z[G])$.
Moreover,  the exact triangle \eqref{eqn:flat_triangle} and the above considerations show that the natural map
\[
R\Gamma_{\fl}(U_{L}, \mu_{p^{\infty}})(p)^{-} \longrightarrow R\Gamma_{\fl}(U_{L}, \mathbb{G}_{m})(p)^{-}
\]
is in fact an isomorphism in $\mathcal{D}(\Z_{p}[G]_{-})$. 
Therefore we obtain the second required isomorphism in $\mathcal{D}(\Z_{p}[G]_{-})$.
\end{proof}

\section{Proof of Theorem \ref{thm:EIMC-implies-BS}: working at the infinite level}\label{sec:proof-infinite-level}

\subsection{Setup and notation}

We now work at the `infinite level' and use the techniques of Iwasawa theory.
Let $\mathcal{G} := \Gal(L_{\infty}/K)$, which we may write as $\mathcal{G} = H \rtimes \Gamma$ where $\Gamma \simeq \Z_{p}$ 
and $H := \Gal(L_{\infty} / K_{\infty})$ naturally identifies with a normal subgroup of $G$.
Let $\Gamma_{0}$ be an open subgroup of $\Gamma$ that is central in $\mathcal{G}$ and recall from 
 \eqref{eq:Lambda-R-decomp} that $\Lambda(\mathcal{G}):=\Z_{p} \llbracket \mathcal{G} \rrbracket $ is a free
$R := \Z_{p} \llbracket \Gamma_{0} \rrbracket $-order in $\mathcal{Q}(\mathcal{G})$.
Let $j \in \mathcal{G}$ denote complex conjugation (this an abuse of notation because its image in the quotient group
$G:=\Gal(L/K)$ is also denoted by $j$) and let $\mathcal{G}^{+} := \mathcal{G} / \langle j \rangle = \Gal(L_{\infty}^{+}/K)$.
Then $j \in H$ and so again $\Lambda(\mathcal{G}^{+})$ is a free $R$-order in $\mathcal{Q}(\mathcal{G}^{+})$. 
Moreover, $\Lambda(\mathcal{G})_{-} := \Lambda(\mathcal{G}) / (1+j)$ is also a free $R$-module of finite rank.
For any $\Lambda(\mathcal{G})$-module $M$ we write $M^{+}$ and $M^{-}$ for the submodules of $M$ upon 
which $j$ acts as $1$ and $-1$, respectively, and consider these as modules over $\Lambda(\mathcal{G}^{+})$ and
$\Lambda(\mathcal{G})_{-}$, respectively. 
We note that $M$ is $R$-torsion if and only if both $M^{+}$ and $M^{-}$ are $R$-torsion.
Furthermore, any $R$-module that is finitely generated as a $\Z_{p}$-module is necessarily $R$-torsion.

Let $\chi_{\mathrm{cyc}}:\mathcal{G} \rightarrow \Z_{p}^{\times}$ denote the $p$-adic cyclotomic character (recall the assumption that $\zeta_{p} \in L$).
Let $\mu_{p^{n}}=\mu_{p^{n}}(L_{\infty})$ denote the group of $p^{n}$th roots of unity in $L_{\infty}^{\times}$
and let $\mu_{p^{\infty}}$ be the nested union (or direct limit) of these groups.
Let $\Z_{p}(1):= \varprojlim_{n} \mu_{p^{n}}$ be endowed with the action of $\mathcal{G}$ given by 
$\chi_{\mathrm{cyc}}$.
For any $r \geq 0$ define $\Z_{p}(r) := \Z_{p}(1)^{\otimes r}$ and $\Z_{p}(-r) := \Hom_{\Z_{p}}(\Z_{p}(r),\Z_{p})$
endowed with the naturally associated actions. 
For any $\Lambda(\mathcal{G})$-module $M$, we define the $r$th Tate twist to be $M(r):= \Z_{p}(r) \otimes_{\Z_{p}} M$
with the natural $\mathcal{G}$-action; hence $M(r)$ is simply $M$ with the modified $\mathcal{G}$-action 
$g \cdot m = \chi_{\mathrm{cyc}}(g)^{r} g(m)$ for $g \in \mathcal{G}$ and $m \in M$.
In particular, we have $\Q_{p} / \Z_{p} (1) \simeq \mu_{p^{\infty}}$ and $\Lambda(\mathcal{G}^{+})(-1) \simeq \Lambda(\mathcal{G})_{-}$.
Recall that for a $\Z$-module $M$ we previously defined $M(p):=\Z_{p} \otimes_{\Z} M$; 
we shall use both notations $M(p)$ and $M(r)$ in the sequel, believing the meaning to be clear from context. 
We note that the property of being $R$-torsion is preserved under taking Tate twists. 

For every place $v$ of $K$ we denote the decomposition subgroup of $\mathcal{G}$ at a chosen prime $w_{\infty}$
above $v$ by $\mathcal{G}_{w_{\infty}}$ (everything will only depend on $v$ and not on $w_{\infty}$ in the following).
We note that the index $[\mathcal{G}:\mathcal{G}_{w_{\infty}}]$ is finite when $v$ is a finite place of $K$.

\subsection{Complexes at the infinite level}

Let $L_{n}$ be the $n$th layer in the cyclotomic $\Z_{p}$-extension of $L$.
Then $\varinjlim_{n} C^{T \bullet}(L_{n}/K)$ defines a complex $C^{T \bullet}(L_{\infty}/K)$ in $\mathcal{D}(\Lambda(\mathcal{G})_{-})$.
We define $A_{L_{\infty}} := \varinjlim_{n} A_{L_{n}}$.

\begin{lemma}\label{lem:cohomology_of_C^T}
We have isomorphisms
\[
H^{i}(C^{T \bullet}(L_{\infty}/K)) \simeq \left\{ \begin{array}{lll}
\mu_{p^{\infty}} \simeq \Q_{p} / \Z_{p} (1) & \mbox{ if } & i=0 \\
A_{L_{\infty}} & \mbox{ if } & i=1 \\
0 & \mbox{ if } & i\not= 0,1.
\end{array} \right.
\]
\end{lemma}

\begin{proof}
Recall that $C^{T \bullet}(L_{n}/K)$ is defined by the middle arrow of the sequence \eqref{eqn:minus-rayclass-sequence} for the layer $L_{n}$.
Taking the direct limit over all $n$ gives an exact sequence of $\Lambda(\mathcal{G})_{-}$-modules
\begin{equation}\label{eqn:limit-sequence}
0 \longrightarrow \mu_{p^{\infty}} \longrightarrow \left( \bigoplus_{v \in T} \ind_{\mathcal{G}_{w_{\infty}}}^{\mathcal{G}} \mu_{p^{\infty}} \right)^{-}
 \longrightarrow A_{L_{\infty}}^{T} \longrightarrow A_{L_{\infty}} \longrightarrow 0, 
\end{equation}
where $A_{L_{\infty}}^{T} := \varinjlim_{n} A_{L_{n}}^{T}$.
Thus $C^{T \bullet}(L_{\infty}/K)$ is the complex in degrees $0$ and $1$ given by the middle arrow of \eqref{eqn:limit-sequence}, 
giving the desired result. 
\end{proof}

For every complex $C^{\bullet}$ in $\mathcal{D}(\Lambda(\mathcal{G}))$
we put $C^{\bullet -} := \Lambda(\mathcal{G})_{-} \otimes^{\mathbb{L}}_{\Lambda(\mathcal{G})} C^{\bullet} \in \mathcal{D}(\Lambda(\mathcal{G})_{-})$.
For a finite set $S$ of places of $K$ we let $U_{\infty, S} := \Spec(\mathcal{O}_{L_{\infty}, S \cup S_{\infty}})$
and put $U_{L_{\infty}} := U_{\infty, S_{\infty}}$.
The following proposition can be viewed as a `derived version' of results that are well-known at the level of cohomology.
It seems possible that this result is known to experts, but the authors were unable to locate a proof in the literature.

\begin{prop} \label{prop:coh-interpretation-Linfty/K}
There are isomorphisms
\[
C^{T \bullet}(L_{\infty}/K)  \simeq R\Gamma_{\fl}(U_{L_{\infty}}, \mu_{p^{\infty}})^{-}
\simeq R\Gamma_{\et}(U_{\infty, S_{p}}, \mu_{p^{\infty}})^{-} 
\]
in $\mathcal{D}(\Lambda(\mathcal{G})_{-})$. In particular, the isomorphism class of
$C^{T \bullet}(L_{\infty}/K)$ does not depend on $T$.
\end{prop}

\begin{proof}
The assumption that $\zeta_{p} \in L$ is crucial in this proof.
Even though it is not strictly necessary, we first check that the two complexes $C^{T \bullet}(L_{\infty}/K)$ and
$R\Gamma_{\et}(U_{\infty, S_{p}}, \mu_{p^{\infty}})^{-} $ compute the same cohomology.
Let $M_{S_{p}}$ be the maximal profinite extension of $L_{\infty}$ that is unramified outside $S_{p}$ and let $M_{S_{p}}^{\ab}(p)$ be the maximal abelian
pro-$p$-extension of $L_{\infty}$ inside $M_{S_{p}}$.
We put $H_{S_{p}} := \Gal(M_{S_{p}} / L_{\infty})$ and $X_{S_{p}} := \Gal(M_{S_{p}}^{\ab}(p) / L_{\infty})$.
There is a canonical isomorphism with Galois cohomology
\begin{equation} \label{eqn:etale-to-Galois}
R\Gamma_{\et}(U_{L_{\infty}, S_{p}}, \mu_{p^{\infty}}) \simeq R\Gamma(H_{S_{p}}, \mu_{p^{\infty}}).
\end{equation}
The strict cohomological $p$-dimension of $H_{S_{p}}$ equals $2$ by \cite[Corollary 10.3.26]{MR2392026} and thus
$H^{i}(H_{S_{p}}, \mu_{p^{\infty}}) = 0$ for all $i \not= 0,1,2$. As the weak Leopoldt conjecture holds for the cyclotomic
$\Z_{p}$-extension, we also have $H^{2}(H_{S_{p}}, \mu_{p^{\infty}}) = H^{2}(H_{S_{p}}, \Q_{p} / \Z_{p})(1) = 0$
by \cite[Theorem 11.3.2]{MR2392026}. We clearly have $H^{0}(H_{S_{p}}, \mu_{p^{\infty}}) = \mu_{p^{\infty}}$. Finally
\[
H^{1}(H_{S_{p}}, \mu_{p^{\infty}})^{-} =  \Hom(H_{S_{p}}, \mu_{p^{\infty}})^{-} =  \Hom(X_{S_{p}}, \mu_{p^{\infty}})^{-}
= \Hom(X_{S_{p}}^{+}, \mu_{p^{\infty}}) \simeq  A_{L_{\infty}},
\]
where the last isomorphism is Kummer duality \cite[Theorem 11.4.3]{MR2392026}. 
If we compare this with Lemma \ref{lem:cohomology_of_C^T},
we see that $R\Gamma_{\et}(U_{\infty, S_{p}}, \mu_{p^{\infty}})^{-}$ and $C^{T \bullet}(L_{\infty}/K)$
compute the same cohomology.

We now establish the derived version of this result and also consider flat cohomology. 
By Proposition \ref{prop:coh-interpretation-L/K} we have isomorphisms 
\[
	C^{T \bullet}(L_{n}/K) \simeq R\Gamma_{\fl}(U_{L_{n}}, \mu_{p^{\infty}})^{-}
\]
for each layer $L_{n}$ in the cyclotomic $\Z_{p}$-extension. Taking direct limits over all $n$ yields the first
required isomorphism by \cite[Chapter III, Lemma 1.16]{MR559531} which holds for the flat topology as well
(see \cite[Chapter III, Remark 1.17 (d)]{MR559531} and \cite[p.~172]{MR244271}). Finally, the natural map
\[
	R\Gamma_{\fl}(U_{L_{\infty}}, \mu_{p^{\infty}}) \longrightarrow
        R\Gamma_{\et}(U_{\infty, S_{p}}, \mu_{p^{\infty}})
\]
is an isomorphism by \cite[Part II, Lemma 3]{MR689645} (note that this result is formulated only in terms
of cohomology, but the proof actually shows that the cone of this map is acyclic).
 \end{proof}

For a finite set $S$ of places of $K$ containing $S_{p} \cup S_{\infty}$ recall the definition of the complex
\[
C_{S}^{\bullet}(L_{\infty}^{+}/K) := R\Gamma_{\et}(\Spec(\mathcal{O}_{L_{\infty}^{+},S}), \Q_{p} / \Z_{p})^{\vee}
\in \mathcal{D}(\Lambda(\Gal(L_{\infty}^{+}/K)))
\]
which occurs in the EIMC. 
For an integer $m$ we let $C_{S}^{\bullet}(L_{\infty}^{+}/K)(m) := \Z_{p}(m) \otimes^{\mathbb{L}}_{\Z} C_{S}^{\bullet}(L_{\infty}^{+}/K)$
be the $m$-fold Tate twist.

\begin{corollary}\label{cor:isom-cT-complex}
We have $C_{S_{p}}^{\bullet}(L_{\infty}^{+}/K)(-1) \simeq C^{T \bullet}(L_{\infty}/K)^{\vee}$ in $\mathcal{D}(\Lambda(\mathcal{G})_{-})$. 
\end{corollary}

\begin{lemma}\label{lem:IT-proj-dim-at-most-one}
Let $T'$ be a non-empty finite set of places of $K$ disjoint from $S_{p} \cup S_{ram} \cup S_{\infty}$
and let 
$
I_{T'}:=\left( \bigoplus_{v\in T'} \ind_{\mathcal{G}_{w_{\infty}}}^{\mathcal{G}} \Z_{p}(-1)\right)^{-}.
$
Then 
\begin{enumerate}
\item $I_{T'}$ is an $\Lambda(\mathcal{G})_{-}$-module of projective dimension at most $1$,
\item for any $\Lambda(\mathcal{G})_{-}$-module $M$ we have $\Ext^{i}_{\Lambda(\mathcal{G})_{-}} (I_{T'}, M)=0$ for $i \geq 2$, and
\item $I_{T'}$ is $R$-torsion.
\end{enumerate}
\end{lemma}

\begin{proof}
For each $v \in T'$ we have an exact sequence of $\Lambda(\mathcal{G}_{w_{\infty}})$-modules
\begin{equation}\label{eqn:resolution-zp(-1)}
0 \longrightarrow \Lambda(\mathcal{G}_{w_{\infty}}) \longrightarrow \Lambda(\mathcal{G}_{w_{\infty}}) \longrightarrow \Z_{p}(-1) \longrightarrow 0,
\end{equation}
where the injection is right multiplication by $1 - \chi_{\mathrm{cyc}}(\phi_{w_{\infty}}) \phi_{w_{\infty}}$
and $\phi_{w_{\infty}}$ denotes the Frobenius automorphism at $w_{\infty}$. 
Claim (i) now follows easily, and this implies claim (ii).
For each $v \in T'$ the index $[\mathcal{G}:\mathcal{G}_{w_{\infty}}]$ is finite; 
thus $I_{T'}$ is finitely generated as a $\Z_{p}$-module, giving claim (iii).
\end{proof}

\begin{prop}\label{prop:construction-of-YST(-1)}
There exists a $\Lambda(\mathcal{G})_{-}$-module $Y_{S}^{T} (-1)$ and a commutative diagram
\begin{equation}\label{eqn:complex-diagram} 
\xymatrix{
{0} \ar[r] & X_{S}^{+} (-1) \ar[r] \ar[d] & Y_{S}^{T} (-1) \ar[r] \ar[d] &
I_{T} \ar[r] \ar@{=}[d] & \Z_{p}(-1) \ar[r] \ar@{=}[d] & 0\\
{0} \ar[r] & X_{S_{p}}^{+} (-1) \ar[r] \ar[d] &  \Hom(A_{L_{\infty}}^{T}, \Q_{p} / \Z_{p}) \ar[r] \ar[d] &
I_{T} \ar[r] &  \Z_{p}(-1) \ar[r] & 0,\\
& 0 & 0 & & &
}\end{equation}
with exact rows and columns
where the middle two terms of upper and lower rows (concentrated in degrees $-1$ and $0$) represent $C_{S}^{\bullet}(L_{\infty}^{+}/K)(-1)$ and $C_{S_{p}}^{\bullet}(L_{\infty}^{+}/K)(-1)$,
respectively.
\end{prop}

\begin{proof}
Using \eqref{eqn:limit-sequence} and Corollary \ref{cor:isom-cT-complex}
we see that for every finite set $T$ of places of $K$ such that $\Hyp(S_{p} \cup S_{\ram} \cup S_{\infty},T)$ is satisfied,
the extension class of the complex $C_{S_{p}}^{\bullet}(L_{\infty}^{+}/K)(-1)$ in
$\Ext^{2}_{\Lambda(\mathcal{G})_{-}} (\Z_{p}(-1), X_{S_{p}}^{+} (-1))$ may be represented by the exact sequence
\begin{equation}\label{eqn:ext-rep-I_T}
0 \longrightarrow X_{S_{p}}^{+} (-1) \longrightarrow \Hom(A_{L_{\infty}}^{T}, \Q_{p} / \Z_{p}) \longrightarrow
I_{T} \longrightarrow \Z_{p}(-1) \longrightarrow 0.
\end{equation}

Let $Z_{S}$ be the kernel of the natural surjection $X_{S}^{+}(-1) \twoheadrightarrow X_{S_{p}}^{+}(-1)$
and let $W_{T}$ be the kernel of the right-most surjection in \eqref{eqn:ext-rep-I_T}.
Then the short exact sequences
\[
0 \longrightarrow W_{T} \longrightarrow I_{T} \longrightarrow \Z_{p}(-1) \longrightarrow 0, \qquad 
0 \longrightarrow Z_{S} \longrightarrow X_{S}^{+}(-1) \longrightarrow X_{S_{p}}^{+}(-1) \longrightarrow 0,
\]
induce long exact sequences in cohomology, which by Lemma \ref{lem:IT-proj-dim-at-most-one} (ii) give the following diagram 
\[
\xymatrix{
\Ext^{1}(W_{T},Z_{S}) \ar[r] \ar[d] &
\Ext^{2}(\Z_{p}(-1),Z_{S}) \ar[r] \ar[d] & 0\\
\Ext^{1}(W_{T},X_{S}^{+}(-1)) \ar[r]^{\alpha_{S}} \ar[d]^{\gamma} & \Ext^{2}(\Z_{p}(-1),X_{S}^{+}(-1)) \ar[r] \ar[d]^{\beta} & 0\\
\Ext^{1}(W_{T},X_{S_{p}}^{+}(-1)) \ar[r]^{\alpha_{p}} \ar[d] & \Ext^{2}(\Z_{p}(-1),X_{S_{p}}^{+}(-1))  \ar[r] \ar[d] & 0\\
\Ext^{2}(W_{T},Z_{S}) \ar[r]^{\sim}  & \Ext^{3}(\Z_{p}(-1),Z_{S}), &
}
\]
where we have omitted the subscript $\Lambda(\mathcal{G})_{-}$ from all $\Ext$-groups, 
and all rows and columns are exact. 
Note that the top two squares are commutative (see \cite[(7.3), p.\ 140]{MR1438546}) and that the bottom square is anti-commutative
(see  \cite[Exercise 9.9, p.\ 156]{MR1438546}). 
By \cite[Lemma 2.4]{MR2371375} $\beta$ maps (the class of) $C_{S}^{\bullet}(L_{\infty}^{+}/K)(-1)$ to 
$C_{S_{p}}^{\bullet}(L_{\infty}^{+}/K)(-1)$.
Let 
\[
\varepsilon := [ 0 \longrightarrow X_{S_{p}}^{+} (-1) \longrightarrow \Hom(A_{L_{\infty}}^{T}, \Q_{p} / \Z_{p}) \longrightarrow
W_{T} \longrightarrow 0] \in \Ext_{\Lambda(\mathcal{G})_{-}}^{1}(W_{T},X_{S_{p}}^{+}(-1)).
\]
Then the fact that $C_{S_{p}}^{\bullet}(L_{\infty}^{+}/K)(-1)$ is represented by 
\eqref{eqn:ext-rep-I_T} shows that $\alpha_{p}$ maps $\varepsilon$ to $C_{S_{p}}^{\bullet}(L_{\infty}^{+}/K)(-1)$.
Now a diagram chase shows that there exists a preimage of $\varepsilon$ under $\gamma$ that is mapped to 
$C_{S}^{\bullet}(L_{\infty}^{+}/K)(-1)$ by $\alpha_{S}$.
\end{proof}

The proof of the following lemma explains why we write $Y_{S}^{T} (-1)$ rather than $Y_{S}^{T}$.
Note that every $\Lambda(\mathcal{G})_{-}$-module $M$ may be written as an $n$-fold Tate twist for every $n \in \Z$;
simply write $M = M(-n)(n)$.

\begin{lemma}\label{lem:YST(-1)-has-proj-dim-at-most-one}
The projective dimension of the $\Lambda(\mathcal{G})_{-}$-module $Y_{S}^{T} (-1)$ is at most $1$.
Moreover, $Y_{S}^{T} (-1)$ is $R$-torsion.
\end{lemma}

\begin{proof}
Let $T \subseteq T''$ be a second finite set of places of $K$ such that $\Hyp(S,T'')$ is satisfied and let $T':=T''-T$.
The short exact sequence
\[
0 \longrightarrow W_{T} \longrightarrow W_{T''} \longrightarrow I_{T'} \longrightarrow 0
\]
induces a long exact sequence in cohomology which by Lemma \ref{lem:IT-proj-dim-at-most-one} (ii) becomes
\[
\Ext^{1}_{\Lambda(\mathcal{G})_{-}}(W_{T''},X_{S}^{+}(-1)) 
\longrightarrow \Ext^{1}_{\Lambda(\mathcal{G})_{-}}(W_{T},X_{S}^{+}(-1)) \longrightarrow 0.
\]
Thus using top row of \eqref{eqn:complex-diagram} for $T$ and $T''$
we have a commutative diagram 
\begin{equation*}\label{eqn:change-of-Ts} 
\xymatrix{
{0} \ar[r] & X_{S}^{+} (-1) \ar[r] \ar@{=}[d] & Y_{S}^{T}(-1) \ar[r] \ar[d] &
I_{T} \ar[r] \ar[d] & \Z_{p}(-1) \ar[r] \ar@{=}[d] & 0\\
{0} \ar[r] & X_{S}^{+} (-1) \ar[r] &  Y_{S}^{T''}(-1)  \ar[r] &
I_{T''} \ar[r] &  \Z_{p}(-1) \ar[r] & 0.
}
\end{equation*}
Applying the snake lemma now gives a short exact sequence
\[
0 \longrightarrow Y_{S}^{T} (-1) \longrightarrow Y_{S}^{T''} (-1) \longrightarrow I_{T'} \longrightarrow 0.
\]
Therefore Lemma \ref{lem:IT-proj-dim-at-most-one} shows that the claim does not depend on the particular choice of $T$ and so by Remark \ref{rmk:conditions-on-T} we can and do assume that $T$ consists of a single place $v$.

Recall that $\mathcal{G}^{+} := \mathcal{G} / \langle j \rangle= \Gal(L_{\infty}^{+}/K)$ and let $\mathcal{G}^{+}_{w_{\infty}^{+}}$
denote the decomposition subgroup at $w_{\infty}^{+}$,
where $w_{\infty}^{+}$ denotes the place of $L_{\infty}^{+}$ below $w_{\infty}$. 
Then we have an isomorphism of $\Lambda(\mathcal{G}^{+})$-modules
\begin{equation} \label{eqn:ind-Zp-iso}
(\ind_{\mathcal{G}_{w_{\infty}}}^{\mathcal{G}} \Z_{p}(-1))^{-}(1) \simeq \ind_{\mathcal{G}^{+}_{w^{+}_{\infty}}}^{\mathcal{G}^{+}} \Z_{p}.
\end{equation}

Let $\Delta(\mathcal{G}^{+})$ denote the kernel of the augmentation map $\Lambda(\mathcal{G}^{+}) \twoheadrightarrow \Z_{p}$.
Given a $\Lambda(\mathcal{G}^{+})$-monomorphism $\psi:\Lambda(\mathcal{G}^{+}) \rightarrow \Delta(\mathcal{G}^{+})$, 
Ritter and Weiss \cite[\S 4]{MR2114937} construct a four term exact sequence whose class in 
$\Ext^{2}_{\Lambda(\mathcal{G}^{+})}(\Z_{p},X_{S}^{+})$ is in fact independent of the choice of $\psi$.
To make this construction explicit, we now choose $\psi$ to be given by right multiplication with $(1 - \phi_{w_{\infty}^{+}})$, 
where $\phi_{w_{\infty}^{+}}$ denotes the Frobenius automorphism at $w_{\infty}^{+}$.
Let $\tilde{\psi}$ be $\psi$ followed by the inclusion $\Delta(\mathcal{G}^{+}) \subset \Lambda(\mathcal{G}^{+})$.
Then the cokernel of $\tilde{\psi}$ identifies with $\ind_{\mathcal{G}^{+}_{w^{+}_{\infty}}}^{\mathcal{G}^{+}} \Z_{p}$, 
and the Ritter and Weiss construction gives a four term exact sequence
\[
0 \longrightarrow X_{S}^{+} \longrightarrow \tilde{Y}^{\{v\}}_{S} \longrightarrow \ind_{\mathcal{G}^{+}_{w^{+}_{\infty}}}^{\mathcal{G}^{+}} \Z_{p}
\longrightarrow \Z_{p} \longrightarrow 0,
\]
where the $\Lambda(\mathcal{G}^{+})$-module $\tilde{Y}^{\{v\}}_{S}$ has projective dimension at most $1$
and is $R$-torsion.

It follows from \cite[Theorem 2.4]{MR3072281}, Corollary \ref{cor:isom-cT-complex} and \eqref{eqn:ind-Zp-iso} 
that we have a commutative diagram
\[ 
\xymatrix{
{0} \ar[r] & X_{S}^{+} \ar[r] \ar@{=}[d] & Y_{S}^{\{v\}}  \ar[r] \ar[d] &
\ind_{\mathcal{G}^{+}_{w^{+}_{\infty}}}^{\mathcal{G}^{+}} \Z_{p} \ar[r] \ar@{=}[d] & \Z_{p} \ar[r] \ar@{=}[d] & 0\\
{0} \ar[r] & X_{S}^{+} \ar[r] &  \tilde{Y}^{\{v\}}_{S} \ar[r] &
\ind_{\mathcal{G}^{+}_{w^{+}_{\infty}}}^{\mathcal{G}^{+}} \Z_{p} \ar[r] &  \Z_{p} \ar[r] & 0.
}
\]
Hence the vertical arrow must be an isomorphism, and thus the projective dimension of the $\Lambda(\mathcal{G}^{+})$-module
$Y_{S}^{\{v\}}$ is at most $1$ and $Y_{S}^{\{v\}}$ is $R$-torsion.
Therefore the same claims are true of the $\Lambda(\mathcal{G})_{-}$-module $Y_{S}^{\{v\}} (-1)$.
\end{proof}

\subsection{Consequences in terms of non-commutative Fitting invariants}\label{subsec:conseq-noncomm-Fitt}
We will henceforth assume that all ramified places belong to $S$. 
For $v \in T$ we put
\[
\xi_{v} := \nr(1 - \chi_{\mathrm{cyc}}(\phi_{w_{\infty}}) \phi_{w_{\infty}}).
\]
We put
\[
\Psi_{S,T} = \Psi_{S,T}(L_{\infty} / K) :=  t_{\mathrm{cyc}}^{1}(\Phi_{S}) \cdot \prod_{v \in T} \xi_{v}.
\]
Note that this slightly differs from the corresponding element $\Psi_{S,T}$ in \cite{MR3072281}.

\begin{prop}\label{prop:EIMC-gives-Fitting}
Suppose that the EIMC holds for $L_{\infty}^{+}/K$. 
Then $\Psi_{S,T}$ is a generator of $\Fitt_{\Lambda(\mathcal{G})_{-}}(Y_{S}^{T} (-1))$.
\end{prop}

\begin{proof}
Since the EIMC holds, by definition \eqref{eqn:fitt-of-complex} we have that 
$\Phi_{S}^{-1}$ generates the Fitting invariant of $C_{S}^{\bullet}(L_{\infty}^{+}/K) \in \mathcal{D}^{\perf}\tor(\Lambda(\mathcal{G}^{+}))$.
However, $t_{\mathrm{cyc}}^{1}$ induces an isomorphism $\Lambda(\mathcal{G}^{+})(-1) \simeq \Lambda(\mathcal{G})_{-}$ and so 
$t_{\mathrm{cyc}}^{1}(\Phi_{S})^{-1}$ generates the Fitting invariant of $C_{S}^{\bullet}(L_{\infty}^{+}/K)(-1) \in \mathcal{D}^{\perf}\tor(\Lambda(\mathcal{G})_{-})$.
Lemmas \ref{lem:IT-proj-dim-at-most-one} and \ref{lem:YST(-1)-has-proj-dim-at-most-one} show that 
$I_{T}$ and $Y_{S}^{T} (-1)$ are both $R$-torsion $\Lambda(\mathcal{G})_{-}$-modules of projective dimension at most $1$; 
hence they both have quadratic presentations by Remark \ref{rmk:M-admits-quadratic-presentation} 
(or alternatively, \cite[Lemma 6.2]{MR2609173}).
Therefore combining Proposition \ref{prop:construction-of-YST(-1)}
and Lemma \ref{lem:fitt-eq-complex-quad} gives
\[
\Fitt_{\Lambda(\mathcal{G})_{-}}(Y_{S}^{T} (-1))  =  \Fitt_{\Lambda(\mathcal{G})_{-}}\left(C_{S}^{\bullet}(L_{\infty}^{+}/K)(-1)\right)^{-1}
\cdot \Fitt_{\Lambda(\mathcal{G})_{-}} (I_{T}).
\]
The exact sequence \eqref{eqn:resolution-zp(-1)} shows that each $(\ind_{\mathcal{G}_{w_{\infty}}}^{\mathcal{G}} \Z_{p}(-1))^{-}$
has a quadratic presentation and that its Fitting invariant is generated by $\xi_{v}$.
Hence Lemma \ref{lem:Fitting-properties} (ii) gives
\[
\Fitt_{\Lambda(\mathcal{G})_{-}}(I_{T}) = 
\prod_{v\in T} \Fitt_{\Lambda(\mathcal{G})_{-}}\left((\ind_{\mathcal{G}_{w_{\infty}}}^{\mathcal{G}} \Z_{p}(-1))^{-}\right) = 
\prod_{v \in T} \left[ \langle  \xi_{v} \rangle_{\zeta(\Lambda(\mathcal{G})_{-})}\right]_{\nr(\zeta(\Lambda(\mathcal{G})_{-})},
\]
and we therefore obtain the desired result.
\end{proof}

We now suppose that the EIMC holds for $L_{\infty}^{+}/K$. The surjection
\[
Y_{S}^{T} (-1) \longrightarrow \Hom(A_{L_{\infty}}^{T}, \Q_{p} / \Z_{p}) \longrightarrow 0
\]
in diagram \eqref{eqn:complex-diagram}, together with Lemma \ref{lem:Fitting-properties} (i) and Proposition \ref{prop:EIMC-gives-Fitting} then imply that
\[
\Psi_{S,T} \in \Fitt_{\Lambda(\mathcal{G})_{-}}^{\max}(\Hom(A_{L_{\infty}}^{T}, \Q_{p} / \Z_{p})).
\]
As the transition maps in the direct limit $A_{L_{\infty}}^{T} = \varinjlim_{n} A_{L_{n}}^{T}$ are injective
by \cite[Lemma 2.9]{MR3383600}, the transition maps in the projective limit
$\Hom(A_{L_{\infty}}^{T}, \Q_{p} / \Z_{p}) = \varprojlim_{n} (A_{L_{n}}^{T})^{\vee}$ are surjective.
As $\Gamma_{L}:=\Gal(L_{\infty}/L)$ clearly acts trivially on $(A_{L}^{T})^{\vee}$, we have a surjection
\begin{equation} \label{eqn:epi-descent}
\Hom(A_{L_{\infty}}^{T}, \Q_{p} / \Z_{p})_{\Gamma_{L}} \longrightarrow (A_{L}^{T})^{\vee} \longrightarrow 0.
\end{equation}
Fix an odd character $\chi  \in \Irr_{\Q_{p}^{c}}(G)$ and view $\chi$ as an irreducible character of $\mathcal{G}$ with open kernel.
We have 
\[
\phi(j_{\chi}(t_{\mathrm{cyc}}^{1}(\Phi_{S}))) = \phi(j_{\chi \omega}^{1}(\Phi_{S})) = L_{p,S}(0, \chi \omega),
\]
where the first and second equalities follow from Lemma \ref{lem:jchi-tcyc-composition} and  \eqref{eq:PhiS-jr-p-adic}, respectively.
Moreover, $\phi(j_{\chi}(\prod_{v \in T} \xi_{v})) = \delta_{T}(0,\check \chi)$.
As Fitting invariants behave well under base change by Proposition \ref{prop:Fitting-descent}, we have
\[
(\theta_{p,S}^{T})^{\sharp} = \sum_{\chi \in \Irr_{\Q_{p}^{c}}(G)} \phi(j_{\chi}(\Psi_{S,T})) e(\chi) \in
\Fitt_{\Z_{p}[G]_{-}}^{\max}(\Hom(A_{L_{\infty}}^{T}, \Q_{p} / \Z_{p})_{\Gamma_{L}}) \subseteq
\Fitt_{\Z_{p}[G]_{-}}^{\max}((A_{L}^{T})^{\vee}),
\]
where we have again used Lemma \ref{lem:Fitting-properties} (i).
This completes the proof of Theorem \ref{thm:EIMC-implies-BS}.
\end{proof}

\section{Hybrid $p$-adic group rings and Frobenius groups}\label{sec:hybrid-and-frobenius}

\subsection{Hybrid $p$-adic group rings}
We recall material on hybrid $p$-adic group rings from \cite[\S 2]{MR3461042} and \cite[\S 2]{MR3749195}.
We shall sometimes abuse notation by using the symbol $\oplus$ to denote the direct product of rings or orders.

Let $p$ be a prime and let $G$ be a finite group.
For a normal subgroup $N \unlhd G$, let $e_{N} = |N|^{-1}\sum_{\sigma \in N} \sigma$
be the associated central trace idempotent in the group algebra $\Q_{p}[G]$.
Then there is a ring isomorphism $\Z_{p}[G]e_{N} \simeq \Z_{p}[G/N]$.
We now specialise \cite[Definition 2.5]{MR3461042} to the case of $p$-adic group rings
(we shall not need the more general case of $N$-hybrid orders).

\begin{definition}
Let $N \unlhd G$. We say that the $p$-adic group ring $\Z_{p}[G]$ is \emph{$N$-hybrid}
if (i) $e_{N} \in \Z_{p}[G]$ (i.e. $p \nmid |N|$) and (ii) $\Z_{p}[G](1-e_{N})$ is a maximal
$\Z_{p}$-order in $\Q_{p}[G](1-e_{N})$.
\end{definition}

\subsection{Frobenius groups}\label{subsec:frobenius-groups}
We recall the definition and some basic facts about Frobenius groups and then use them to
provide many examples of hybrid group rings.
For further results and examples, we refer the reader to \cite[\S 2.3]{MR3461042} and \cite[\S 2.2]{MR3749195}.

\begin{definition}
A \emph{Frobenius group} is a finite group $G$ with a proper non-trivial subgroup $H$
such that $H \cap gHg^{-1}=\{ 1 \}$ for all $g \in G-H$,
in which case $H$ is called a \emph{Frobenius complement}.
\end{definition}

\begin{theorem}\label{thm:frob-kernel}
A Frobenius group $G$ contains a unique normal subgroup $N$, known as the Frobenius kernel, such that
$G$ is a semidirect product $N \rtimes H$. Moreover:
\begin{enumerate}
\item $|N|$ and $[G:N]=|H|$ are relatively prime.
\item The Frobenius kernel $N$ is nilpotent.
\item If $K \unlhd G $ then either $K \unlhd N$ or $N \unlhd K$.
\item If $\chi \in \Irr_{\C}(G)$ such that  $N \not \leq \ker \chi$ then $\chi= \ind_{N}^{G}(\psi)$ for some $1 \neq \psi \in \Irr_{\C}(N)$.
\end{enumerate}
\end{theorem}

\begin{proof}
For (i) and (iv) see \cite[\S 14A]{MR632548}.
For (ii) see \cite[10.5.6]{MR1357169} and for (iii) see  \cite[Exercise 7, \S 8.5]{MR1357169}.
\end{proof}

\begin{prop}[{\cite[Proposition 2.13]{MR3461042}}]\label{prop:frob-N-hybrid}
Let $G$ be a Frobenius group with Frobenius kernel $N$.
Then for every prime $p$ not dividing $|N|$, the group ring $\Z_{p}[G]$ is $N$-hybrid.
\end{prop}

For $n \in \N$, let $C_{n}$ denote the cyclic group of order $n$, let $A_{n}$ denote the alternating group on $n$ letters and let $S_{n}$ denote the symmetric group on $n$ letters.
Let $V_{4}$ denote the subgroup of $A_{4}$ generated by double transpositions. 
We now recall two examples from \cite[\S 2.3]{MR3461042} (also see \cite[\S 2.2]{MR3749195}).

\begin{example}\label{ex:metacyclic}
Let $p<q$ be distinct primes and assume that $p \mid (q-1)$.
Then there is an embedding $C_{p} \hookrightarrow \Aut(C_{q})$ and so there is
a fixed-point-free action of $C_{p}$ on $C_{q}$.
Hence the corresponding semidirect product $G = C_{q} \rtimes C_{p}$ is a Frobenius group
(see \cite[Theorem 2.12]{MR3461042} or \cite[\S 4.6]{MR2599132}, for example), and so $\Z_{p}[G]$ is $N$-hybrid with $N = C_{q}$.
\end{example}

\begin{example}\label{ex:affine}
Let $q$ be a prime power and let $\F_{q}$ be the finite field with $q$ elements.
The group $\Aff(q)$ of affine transformations on $\F_{q}$ is the group of transformations
of the form $x \mapsto ax +b$ with $a \in \F_{q}^{\times}$ and $b \in \F_{q}$.
Let $G=\Aff(q)$ and let $N=\{ x \mapsto x+b \mid b \in \F_{q} \}$.
Then $G$ is a Frobenius group with Frobenius kernel $N \simeq \F_{q}$ and is isomorphic to
the semidirect product $\F_{q} \rtimes \F_{q}^{\times}$ with the natural action.
Hence for every prime $p$ not dividing $q$, we have that $\Z_{p}[G]$ is $N$-hybrid.
Note that in particular $\Aff(3) \simeq S_{3}$ and $\Aff(4) \simeq A_{4}$.
\end{example}

\begin{lemma}\label{lem:monomial-Frobenius}
Let $G = N \rtimes H$ be a Frobenius group.
Then $G$ is monomial if and only if its Frobenius complement $H$ is monomial.
In particular, if $H$ is supersoluble or metabelian then $G$ is monomial.
\end{lemma}

\begin{proof}
Suppose $H$ is monomial.
Let $\chi \in \Irr_{\C}(G)$.
If $N \leq \ker \chi$ then $\chi$ is inflated from some $\varphi \in \Irr_{\C}(G/N)$.
Otherwise $N \nleq \ker \chi$ and so $\chi$ is induced
from some $\psi \in \Irr_{\C}(N)$ by Theorem \ref{thm:frob-kernel} (iv).
The Frobenius complement $H \simeq G/N$ is monomial by assumption.
Moreover, the Frobenius kernel $N$ is nilpotent by Theorem \ref{thm:frob-kernel} (ii) and thus is monomial 
by \cite[Theorem 11.3]{MR632548}.
However, induction is transitive and inflation commutes with induction (as in \cite[Theorem 4.2 (3)]{MR1984740}, for example)
so that in both cases $\chi$ is induced from a linear character. Therefore $G$ is monomial.
The converse follows from the fact that any quotient of a monomial group is monomial
(this can easily be proved using that inflation commutes with induction; also see \cite[Chapter 2, \S 4]{MR655785}).
The last claim follows since $H$ is monomial in these cases by \cite[\S 4.4, Theorem 4.8 (1)]{MR1984740}).
\end{proof}

\subsection{New $p$-adic hybrid group rings from old}
We recall two results from \cite[\S 2]{MR3749195}.

\begin{prop}[{\cite[Proposition 2.14]{MR3749195}}] \label{prop:hybrid-basechange-down}
Let $G$ be a finite group with normal subgroups $N, H \unlhd  G$.
Let $N'$ be a normal subgroup of $H$ such that $N' \leq N$. Let $p$ be a prime.
If  $\Z_{p}[G]$ is $N$-hybrid then $\Z_{p}[H]$ is $N'$-hybrid.
\end{prop}

\begin{prop}[{\cite[Proposition 2.15]{MR3749195}}]  \label{prop:hybrid-basechange-up}
Let $G$ be a finite group with normal subgroups $N \unlhd H \unlhd G$ such that $N \unlhd G$.
Let $p$ be a prime and assume that $p \nmid [G:H]$.
Then $\Z_{p}[G]$ is $N$-hybrid if and only if $\Z_{p}[H]$ is $N$-hybrid.
\end{prop}

\begin{example}\label{ex:S4-A4-V4}
Let $p=3$, $G=S_{4}$, $H=A_{4}$ and $N = V_{4}$. 
Then the hypotheses of Proposition \ref{prop:hybrid-basechange-up} are satisfied.
Hence $\Z_{3}[S_{4}]$ is $V_{4}$-hybrid if and only if $\Z_{3}[A_{4}]$ is $V_{4}$-hybrid.
In fact, $\Z_{3}[A_{4}]$ is indeed $V_{4}$-hybrid since $A_{4}$ is a Frobenius group with Frobenius kernel $V_{4}$
(see Example \ref{ex:affine}) and so $\Z_{3}[S_{4}]$ is also $V_{4}$-hybrid.
However, $S_{4}$ is \emph{not} a Frobenius group (see \cite[Example 2.18]{MR3461042}).
Thus Proposition \ref{prop:hybrid-basechange-up} can be used to give examples which do not come directly from
Proposition \ref{prop:frob-N-hybrid}.
\end{example}

\begin{remark}\label{rmk:hybrid-Galois-basechange}
One of the main reasons for interest in Propositions \ref{prop:hybrid-basechange-down} and \ref{prop:hybrid-basechange-up}
comes from base change in Galois theory.
Assume that $L/K$ is a finite Galois extension with Galois group $G$.
Let $p$ be a prime and suppose that $\Z_{p}[G]$ is $N$-hybrid for some normal subgroup $N$ of $G$.
Put $F := L^{N}$.
Now let $K'/K$ be a finite abelian extension of $K$ and put $F' = FK'$ and $L' = LK'$.
Then $H := \Gal(L'/K')$ naturally identifies with a subgroup of $G$.
Similarly, $N' := \Gal(L'/F')$ is normal in $H$ and identifies with a subgroup of $N$.
The fixed field $L^{H}$ is a subfield of $K'$ and thus $L^{H}/K$ is a Galois extension, since $K'/K$ is abelian.
Hence $H$ is normal in $G$ and we conclude by Proposition \ref{prop:hybrid-basechange-down}
that $\Z_{p}[H]$ is $N'$-hybrid. Finally, let $G' := \Gal(L'/K)$; if $p \nmid [K':K]$ then
 $\Z_{p}[G']$ is also $N'$-hybrid by Proposition \ref{prop:hybrid-basechange-up}.
The situation is illustrated by the following field diagram.
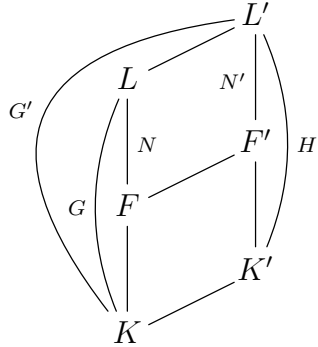
\[
\xymatrix@1@!0@W=14pt@H=14pt{
& & L'  \ar@{-}[dll] \ar@{-}[dd]_{N'} \ar@/^1pc/@{-}[dddd]^{H} \ar@{-}@/_5pc/[dddddll] _{G'} \\
L \ar@{-}[dd]^{N}  \ar@/_1pc/@{-}[dddd]_{G} & &   \\
& & F' \ar@{-}[dll] \ar@{-}[dd]  \\
F \ar@{-}[dd] & & \\
& & K' \ar@{-}[dll] \\
K & &
}
\]
%
\end{remark}

\section{Unconditional results}\label{sec:unconditional}

\subsection{Extensions of degree coprime to $p$}
We first consider a straightforward case.

\begin{theorem}\label{thm:BS-p-nmid-G}
Let $L/K$ be a finite Galois CM-extension of number fields.
Let $p$ be an odd prime and 
let $S$ be a finite set of places of $K$ such that $S_{p} \cup S_{\ram}(L/K) \cup S_{\infty} \subseteq S$.
If $\Gal(L^{+}/K)$ is monomial and $p \nmid [L^{+}:K]$ then both $BS(L/K,S,p)$ and $B(L/K,S,p)$ are true. 
\end{theorem}

\begin{proof}
Since $[L^{+}:K]$ and $[L(\zeta_p)^{+}:L^{+}]$ are both coprime to $p$, their product $[L(\zeta_p)^{+}:K]$ is also coprime to $p$.
Hence by Theorem \ref{thm:EIMC-p-does-not-divide-order-of-H} the EIMC holds for $L(\zeta_p)^{+}_{\infty} / K$, and so 
the desired result follows from Corollary \ref{cor:BS-monomial}.
\end{proof}

\subsection{Further unconditional cases of the EIMC}
We now apply the results of \cite{MR3749195} to give criteria for the EIMC to hold unconditionally in cases of interest to us.

\begin{theorem}\label{thm:pth-root-base-change}
Let $L/K$ be a finite Galois extension of totally real number fields with Galois group $G$.
Let $p$ be an odd prime and let $L_{\infty}$ be the cyclotomic $\Z_{p}$-extension of $L$.
Let $N$ be a normal subgroup of $G$ and let $\overline{P}$ be a Sylow $p$-subgroup of $\overline{G}:=\Gal(L^{N}/K) \simeq G/N$.
Suppose that $\Z_{p}[G]$ is $N$-hybrid and that $(L^{N})^{\overline{P}}/\Q$ is abelian.
Let $K'/K$ be a field extension such that $K'$ is totally real, $K'/\Q$ is abelian and $p \nmid [K':K] < \infty$.
Let $L'_{\infty}=L_{\infty}K'$.
Then the EIMC holds for both $L_{\infty}/K$ and $L_{\infty}'/K$.
\end{theorem}

\begin{remark}\label{rmk:can-take-N-trivial}
It is straightforward to see that for every prime $p$ and every finite group $G$, the $p$-adic group ring $\Z_{p}[G]$ is $\{ 1 \}$-hybrid.
Hence, in particular, Theorem \ref{thm:pth-root-base-change} can be applied in the case that $N$ is trivial.
\end{remark}

\begin{remark}\label{rmk:base-change-in-particular}
The hypothesis that $(L^{N})^{\overline{P}}/\Q$ is abelian forces $K/\Q$ to be abelian, and thus one can
take $K'$ to be the compositum of $K$ with another finite abelian extension $K''/\Q$ such that $p \nmid [K'':\Q]$.
In particular, Theorem \ref{thm:pth-root-base-change} can be applied with $K'=K(\zeta_{p})^{+}$
and $L_{\infty}'=L_{\infty}(\zeta_{p})^{+}=L(\zeta_{p})^{+}_{\infty}$.
\end{remark}

\begin{proof}[Proof of Theorem \ref{thm:pth-root-base-change}]
The EIMC holds for $L_{\infty}/K$ by \cite[Theorem 4.17 (v)]{MR3749195}.
Let $F=L^{N}$ and put $F'=FK'$ and $L'=LK'$.
Let $G'=\Gal(L'/K)$ and $N'=\Gal(L'/F')$.
Then $\Z_{p}[G']$ is $N'$-hybrid by Remark \ref{rmk:hybrid-Galois-basechange}.
Let $\overline{P}'$ be a Sylow $p$-subgroup of $\overline{G'}:=G'/N'$.
Then
\[
((L')^{N'})^{\overline{P}'} = (F')^{\overline{P}'} = F^{\overline{P}} K'= (L^{N})^{\overline{P}}K',
\]
which is an abelian extension of $\Q$
as it is the compositum of two such extensions.
Therefore the EIMC holds for $L_{\infty}'/K$ by a further application of \cite[Theorem 4.17 (v)]{MR3749195}.
\end{proof}

\subsection{Unconditional results on the non-abelian Brumer--Stark conjecture}

\begin{theorem}\label{thm:unconditional-BS}
Let $L/K$ be a finite Galois CM-extension of number fields.
Let $N$ be a normal subgroup of $G:=\Gal(L^{+}/K)$ and let $F=(L^{+})^{N}$.
Let $p$ be an odd prime and let $\overline{P}$ be a Sylow $p$-subgroup of $\overline{G}:=\Gal(F/K) \simeq G/N$.
Suppose that $\Z_{p}[G]$ is $N$-hybrid, $G$ is monomial, and $F^{\overline{P}}/\Q$ is abelian.
Let $S$ be a finite set of places of $K$ such that $S_{p} \cup S_{\ram}(L/K) \cup S_{\infty} \subseteq S$.
Then both $BS(L/K,S,p)$ and $B(L/K,S,p)$ are true.
\end{theorem}

\begin{remark}\label{rmk:N-trivial-in-BS}
In particular, Theorem \ref{thm:unconditional-BS} can be applied in the case that $N$ is trivial and $F=L^{+}$ (see Remark \ref{rmk:can-take-N-trivial}).
\end{remark}

\begin{proof}[Proof of Theorem \ref{thm:unconditional-BS}]
Applying Theorem \ref{thm:pth-root-base-change} with $K'=K(\zeta_{p})^{+}$ gives the EIMC for $L(\zeta_{p})^{+}_{\infty}/K$.
Hence the desired result follows from Corollary \ref{cor:BS-monomial}.
\end{proof}

\begin{corollary}\label{cor:BS-Frob}
Let $L/K$ be a finite Galois CM-extension of number fields and let $G=\Gal(L^{+}/K)$.
Suppose that $G = U \rtimes V$ is a Frobenius group with Frobenius kernel $U$ and abelian Frobenius complement $V$.
Further suppose that $(L^{+})^{U}/\Q$ is abelian (in particular, this is the case when $K=\Q$).
Let $p$ be an odd prime and let $S$ be a finite set of places of $K$ such that $S_{p} \cup S_{\ram}(L/K) \cup S_{\infty} \subseteq S$.
Suppose that either $p \nmid |U|$ or $U$ is a $p$-group (in particular, this is the case if $U$ is an $\ell$-group for any prime $\ell$.)
Then both $BS(L/K,S,p)$ and $B(L/K,S,p)$ are true.
\end{corollary}

\begin{proof}
First note that $G$ is monomial by Lemma \ref{lem:monomial-Frobenius} since $V$ is abelian.
Suppose that $p \nmid |U|$. Let $N=U$ and $F=(L^{+})^{N}$.
Then $\Z_{p}[G]$ is $N$-hybrid by Proposition \ref{prop:frob-N-hybrid}.
Hence the desired result follows from Theorem \ref{thm:unconditional-BS} in this case since $F/\Q$ is abelian, 
which forces $F^{\overline{P}}/\Q$ to be abelian.
Suppose that $U$ is a $p$-group. Taking $N=\{1\}$ and $F=L^{+}$ (see Remark \ref{rmk:N-trivial-in-BS}) we apply
Theorem \ref{thm:unconditional-BS} with $\overline{G}=G$ and $\overline{P}=U$ to obtain the desired result. 
\end{proof}

\begin{example}
In particular, $U$ is an $\ell$-group in Corollary \ref{cor:BS-Frob} in the following cases:
\begin{itemize}
\item $G \simeq \Aff(q)$, where $q$ is a prime power (see Example \ref{ex:affine}),
\item $G \simeq C_{\ell} \rtimes C_{q}$, where $q<\ell$ are distinct primes such that $q \mid (\ell-1)$ and $C_{q}$ acts on $C_{\ell}$ via an embedding $C_{q} \hookrightarrow \Aut(C_{\ell})$ (see Example \ref{ex:metacyclic}),
\item $G$ is isomorphic to any of the Frobenius groups constructed in \cite[Example 2.11]{MR3749195}.
\end{itemize} 
\end{example}

\begin{corollary} \label{cor:BS-holds}
Let $L/K$ be a finite Galois CM-extension of number fields.
Let $p$ be an odd prime and let $S$ be a finite set of places of $K$ such that $S_{p} \cup S_{\ram}(L/K) \cup S_{\infty} \subseteq S$.
Then both $BS(L/K,S,p)$ and $B(L/K,S,p)$ are true when  $L^{+}/K$ is any of the extensions considered in 
\cite[Examples 4.21, 4.22 or 4.23]{MR3749195}.
\end{corollary}

\begin{proof}
The group in \cite[Examples 4.21]{MR3749195} is monomial by an application of Lemma \ref{lem:monomial-Frobenius}
and the group in \cite[Example 4.22]{MR3749195} is monomial by an application of \cite[Chapter 2, Theorem 3.10]{MR655785}. 
For \cite[Examples 4.23]{MR3749195} it is straightforward to check that $S_{4}$ is a monomial group.
That the remaining hypotheses of Theorem \ref{thm:unconditional-BS} are satisfied in each case are verified in the cited examples themselves.
 \end{proof}

In certain situations, we can also remove the condition that $S_{p} \subseteq S$. 
To illustrate this, we conclude with the following result.

\begin{theorem}
Let $L/\Q$ be a finite Galois CM-extension of the rationals.
Suppose that $\Gal(L / \Q) \simeq \langle j \rangle \times G$, where $G =\Gal(L^{+}/\Q) = N \rtimes V$
is a Frobenius group with Frobenius kernel $N$ and abelian Frobenius complement $V$.
Suppose further that $N$ is an $\ell$-group for some prime $\ell$.
Then both $BS(L/\Q,S,p)$ and $B(L/\Q,S,p)$ are true for every odd prime $p$ and every finite set $S$ of places
of $\Q$ such that $ S_{\ram}(L/\Q) \cup S_{\infty} \subseteq S$.
\end{theorem}

\begin{proof}
By Corollary \ref{cor:BS-Frob}, we see that the desired result holds when we assume in addition that $S$ contains $S_{p}=\{ p \}$.
Hence we may assume that $p$ is unramified in $L/\Q$ since  $S_{\ram}(L/\Q) \subseteq S$.
We claim that in this case the $p$-minus part of the equivariant Tamagawa number conjecture (ETNC) holds
for the pair $(h^{0}(\Spec(L)), \Z[\Gal(L/\Q)])$ which implies both $BS(L/K,S,p)$ and $B(L/K,S,p)$ by \cite[Theorem 5.3]{MR2976321}.
As $p$ is unramified in $L/\Q$, we have $L \not\subseteq L^{+}(\zeta_{p})$
and thus the Strong Stark conjecture at $p$ holds for each odd character of $\Gal(L / \Q)$  by \cite[Corollary 2]{MR2805422}.
Moreover, when $p = \ell$ then the relevant $\mu$-invariant vanishes by the theorem of Ferrero and Washington
\cite{MR528968} and \cite[Theorem 11.3.8]{MR2392026}. 
Hence the minus-$p$-part of the ETNC holds by \cite[Theorem 1.3]{MR3552493}. Now suppose that $p \neq \ell$. 
Then the group ring $\Z_{p}[G]$ is $N$-hybrid
by Proposition \ref{prop:frob-N-hybrid} and so $\Z_{p}[\Gal(L/\Q)]$ is also $N$-hybrid by Proposition \ref{prop:hybrid-basechange-up} since $p$ is odd.
As $L^{N} / \Q$ is abelian, the (minus-$p$-part of the) ETNC holds for the pair $(h^{0}(\Spec(L^{N})), \Z[\Gal(L^{N} / \Q)])$
by work of Burns and Greither \cite{MR1992015}. Thus the minus-$p$-part of the ETNC for the pair $(h^{0}(\Spec(L)), \Z[\Gal(L/\Q)])$ 
holds as well by restricting  \cite[Theorem 4.3]{MR3461042} to minus parts.
\end{proof}

\bibliography{non-abelian-Brumer-Stark-Bib}{}
\bibliographystyle{amsalpha}

\end{document}